\documentclass[reqno]{amsart}
\usepackage[margin=1.5in,bottom=1.25in]{geometry}	

\usepackage{amsmath}		
\usepackage{amssymb}		
\usepackage{amsfonts}		
\usepackage{amsthm}		
\usepackage[foot]{amsaddr}		

\usepackage[centercolon=true]{mathtools}	
\mathtoolsset{%
}

\usepackage[utf8]{inputenc}		
\usepackage[T1]{fontenc}		

\usepackage[
cal=cm,
]
{mathalfa}

\usepackage{dsfont}		

\usepackage[sf,mono=false]{libertine}

\usepackage{acronym}		
\newcommand{\acli}[1]{\emph{\acl{#1}}}		
\newcommand{\acdefp}[1]{\emph{\aclp{#1}} \textup{(\acsp{#1})}\acused{#1}}	

\newcommand{\Acli}[1]{\emph{\Acl{#1}}}		

\usepackage[labelfont={bf,small},labelsep=colon,font=small]{caption}	
\usepackage{subcaption}		

\usepackage[dvipsnames,svgnames]{xcolor}		

\colorlet{MyRed}{FireBrick!50!Crimson}
\colorlet{MyBlue}{DodgerBlue!75!black}
\colorlet{MyGreen}{DarkGreen!85!black}
\colorlet{MyViolet}{DarkMagenta!80!Blue}

\colorlet{MyLightBlue}{DodgerBlue!20}
\colorlet{MyLightGreen}{MyGreen!20}

\colorlet{PrimalColor}{MyBlue}
\colorlet{PrimalFill}{MyLightBlue}
\colorlet{DualColor}{MyRed}

\colorlet{RevColor}{MyViolet}
\colorlet{MacroColor}{DarkMagenta}
\colorlet{LinkColor}{MediumBlue}

\newcommand{\afterhead}{}		
\newcommand{\para}[1]{\medskip\paragraph{\textbf{#1\afterhead}}}

\usepackage{cancel}		
\usepackage{latexsym}		

\usepackage{pifont}		

\usepackage{tikz}		
\usetikzlibrary{calc,patterns}		
\usepackage{pgfplots}

\usepackage{array}		
\usepackage{booktabs}		
\usepackage[inline,shortlabels]{enumitem}		
\setenumerate{itemsep=\medskipamount,topsep=\smallskipamount,leftmargin=2\parindent}
\setitemize{itemsep=\medskipamount,topsep=\smallskipamount,leftmargin=2\parindent}

\usepackage[kerning=true]{microtype}		

\usepackage{tabto}		
\usepackage{xspace}		

\usepackage[authoryear,sort&compress]{natbib}		

\bibpunct[, ]{[}{]}{,}{n}{,}{,}

\usepackage{hyperref}
\hypersetup{
colorlinks=true,
linktocpage=true,
pdfstartview=FitH,
breaklinks=true,
pdfpagemode=UseNone,
pageanchor=true,
pdfpagemode=UseOutlines,
plainpages=false,
bookmarksnumbered,
bookmarksopen=false,
bookmarksopenlevel=1,
hypertexnames=true,
pdfhighlight=/O,
urlcolor=LinkColor,linkcolor=LinkColor,citecolor=LinkColor,	
pdftitle={},
pdfauthor={},
pdfsubject={},
pdfkeywords={},
pdfcreator={pdfLaTeX},
pdfproducer={LaTeX with hyperref}
}

\usepackage[sort&compress,capitalize,nameinlink]{cleveref}	

\crefname{algo}{Algorithm}{Algorithms}
\crefname{assumption}{Assumption}{Assumptions}
\crefname{case}{Case}{Cases}
\crefname{section}{Section}{Sections}
\crefname{subsection}{Section}{Sections}

\usepackage{algorithm}		
\usepackage{algpseudocode}		

\usepackage{thm-restate}		

\theoremstyle{plain}
\newtheorem{theorem}{Theorem}	
\newtheorem{corollary}{Corollary}	
\newtheorem{lemma}{Lemma}	
\newtheorem{proposition}{Proposition}	

\newtheorem*{theorem*}{Theorem}	
\newtheorem*{corollary*}{Corollary}	

\theoremstyle{definition}
\newtheorem{definition}{Definition}		
\newtheorem{assumption}{Assumption}		
\newtheorem{example}{Example}		

\newtheorem*{definition*}{Definition}		
\newtheorem*{assumption*}{Assumptions}		
\newtheorem*{example*}{Example}		

\theoremstyle{remark}
\newtheorem{remark}{Remark}	

\newtheorem*{remark*}{Remark}	
\newtheorem*{notation*}{Notation}	

\def\endenv{\ding{166}\medskip}	

\newcounter{proofstep}

\usepackage[suppress]{color-edits}		
\usepackage[normalem]{ulem}		

\newcommand{\draft}[1]{#1}		

\newcommand{\revise}[1]{#1}		

\def\beginrev{}		

\newcommand{\newmacro}[2]{\newcommand{#1}{\draft{#2}}}		
\newcommand{\newop}[2]{\DeclareMathOperator{#1}{\draft{#2}}}		

\DeclarePairedDelimiter{\braces}{\{}{\}}		
\DeclarePairedDelimiter{\bracks}{[}{]}		
\DeclarePairedDelimiter{\parens}{(}{)}		

\DeclarePairedDelimiter{\abs}{\lvert}{\rvert}		
\DeclarePairedDelimiter{\pospart}{[}{]_{+}}		

\DeclarePairedDelimiterX{\setdef}[2]{\{}{\}}{#1:#2}		
\DeclarePairedDelimiterXPP{\exclude}[1]{\mathopen{}\setminus}{\{}{\}}{}{#1}

\newcommand{\R}{\mathbb{R}}		

\DeclareMathOperator*{\argmin}{arg\,min}		

\DeclareMathOperator{\aff}{aff}		
\DeclareMathOperator{\lspan}{span}		
\DeclareMathOperator{\bigoh}{\mathcal{O}}		
\DeclarePairedDelimiterXPP{\bigof}[1]{\mathcal{O}}{(}{)}{}{#1}		
\DeclareMathOperator{\dist}{dist}		
\DeclareMathOperator{\dom}{dom}		
\DeclareMathOperator{\grad}{\nabla}		
\DeclareMathOperator{\intr}{int}		
\DeclareMathOperator{\one}{\mathds{1}}		
\DeclareMathOperator{\relint}{ri}		

\newcommand{\cf}{cf.\xspace}		
\newcommand{\eg}{e.g.,\xspace}		
\newcommand{\ie}{i.e.,\xspace}		
\newcommand{\viz}{viz.\xspace}		

\newcommand{\textpar}[1]{\textup(#1\textup)}		

\newcommand{\alt}[1]{#1'}		
\newcommand{\altalt}[1]{#1''}		

\newmacro{\dd}{\:d}		
\newcommand{\eps}{\varepsilon}		

\newmacro{\const}{c}		
\newmacro{\Const}{\rho}		
\newmacro{\coefalt}{\mu}		
\usepackage{xparse}
\NewDocumentCommand{\coef}{O{\lambda}}{\draft{#1}}
\newmacro{\param}{\theta}		
\newmacro{\params}{\Theta}		

\newmacro{\pexp}{p}		
\newmacro{\qexp}{q}		
\newmacro{\rexp}{r}		

\newmacro{\radius}{r}

\newmacro{\beforestart}{0}		
\newmacro{\start}{1}		
\newmacro{\afterstart}{2}		
\newmacro{\running}{\start,\afterstart,\dotsc}		
\newmacro{\halfrunning}{1,3/2,2\dotsc}		

\newmacro{\run}{t}		
\newmacro{\runalt}{s}		
\newmacro{\runaltalt}{\tau}		
\newmacro{\nRuns}{T}		
\newmacro{\runs}{\mathcal{\nRuns}}		

\newmacro{\state}{x}		
\newmacro{\statealt}{y}		
\newmacro{\statealtalt}{z}		

\newcommand{\new}[1][\point]{#1^{+}}		

\newcommand{\beforeinit}[1][\state]{\draft{#1}_{\beforestart}}		
\newcommand{\init}[1][\state]{\draft{#1}_{\start}}		
\newcommand{\afterinit}[1][\state]{\draft{#1}_{\afterstart}}		

\newcommand{\iter}[1][\state]{\draft{#1}_{\runalt}}		
\newcommand{\iterlead}[1][\state]{\draft{#1}_{\runalt+1/2}}		

\newcommand{\prev}[1][\state]{\draft{#1}_{\run-1}}		
\newcommand{\curr}[1][\state]{\draft{#1}_{\run}}		
\renewcommand{\next}[1][\state]{\draft{#1}_{\run+1}}		

\newcommand{\beforelead}[1][\state]{\draft{#1}_{\run-1/2}}		
\newcommand{\lead}[1][\state]{\draft{#1}_{\run+1/2}}		

\newmacro{\tstart}{0}		
\newmacro{\timealt}{s}		
\newmacro{\horizon}{T}		

\newmacro{\traj}{x}		
\newmacro{\trajalt}{y}		
\newmacro{\trajaltalt}{z}		

\newmacro{\flowmap}{\Theta}		
\DeclarePairedDelimiterXPP{\flowof}[2]{\flowmap_{#1}}{(}{)}{}{#2}		

\newop{\Nash}{NE}		
\newop{\CE}{CE}		
\newop{\CCE}{CCE}		
\newop{\NI}{NI}		

\newop{\brep}{br}		
\newop{\reg}{Reg}		
\newop{\preg}{\overline{Reg}}		
\newop{\val}{val}		

\newcommand{\strats}{\points}		
\newcommand{\eq}{\sol}		

\newmacro{\play}{i}		
\newmacro{\playalt}{j}		
\newmacro{\playaltlalt}{k}		
\newmacro{\nPlayers}{N}		
\newmacro{\players}{\mathcal{\nPlayers}}		

\newmacro{\pure}{\alpha}		
\newmacro{\purealt}{\beta}		
\newmacro{\purealtalt}{\gamma}		
\newmacro{\nPures}{A}		
\newmacro{\pures}{\mathcal{\nPures}}		

\newmacro{\loss}{\ell}		
\newmacro{\pay}{u}		
\newmacro{\payv}{v}		
\newmacro{\pot}{f}		

\newmacro{\game}{\mathcal{G}}		
\newmacro{\gamefull}{\game(\players,\points,\pay)}		

\newmacro{\fingame}{\Gamma}		
\newmacro{\fingamefull}{\Gamma(\players,\pures,\pay)}		

\newmacro{\gmat}{g}		
\newmacro{\gdist}{\dist_{\gmat}}
\newmacro{\mfld}{M}		
\newmacro{\form}{\omega}		

\newmacro{\tvec}{z}		
\newmacro{\uvec}{u}		

\newmacro{\ball}{\basin}		
\newmacro{\sphere}{\mathbb{S}}		

\newmacro{\graph}{\mathcal{G}}
\newmacro{\vertices}{\mathcal{V}}
\newmacro{\edges}{\mathcal{E}}

\newmacro{\mat}{A}		
\newmacro{\matalt}{c}		
\newmacro{\hmat}{H}		

\newop{\row}{row}		
\newop{\col}{col}		

\newmacro{\ones}{\mathbf{1}}		
\newmacro{\eye}{I}		
\newmacro{\zer}{\mathbf{0}}		

\DeclarePairedDelimiter{\norm}{\lVert}{\rVert}		
\DeclarePairedDelimiterXPP{\dnorm}[1]{}{\lVert}{\rVert}{_{\ast}}{#1}		

\DeclarePairedDelimiterXPP{\onenorm}[1]{}{\lVert}{\rVert}{_{1}}{#1}		
\DeclarePairedDelimiterXPP{\twonorm}[1]{}{\lVert}{\rVert}{_{2}}{#1}		
\DeclarePairedDelimiterXPP{\supnorm}[1]{}{\lVert}{\rVert}{_{\infty}}{#1}		

\DeclarePairedDelimiterX{\braket}[2]{\langle}{\rangle}{#1,#2}		

\newmacro{\vecspace}{\mathcal{V}}		
\newmacro{\subspace}{\mathcal{W}}		

\newmacro{\coord}{i}		
\newmacro{\coordalt}{j}		
\newmacro{\coordaltalt}{k}		
\newmacro{\nCoords}{n}		
\newmacro{\dims}{\nCoords}		
\newmacro{\vdim}{\nCoords}		

\newmacro{\pvec}{z}		
\newmacro{\pvecalt}{r}		
\newmacro{\bvec}{e}		
\newmacro{\bvecs}{\mathcal{E}}		
\newmacro{\cvec}{b}     
\newmacro{\cvecalt}{d}     

\newmacro{\pspace}{\mathcal{V}}		
\newmacro{\dspace}{\pspace^{\ast}}		

\newmacro{\dvec}{\dpoint}		
\newmacro{\dbvec}{\eps}		

\newmacro{\dpoint}{y}		
\newmacro{\dpointalt}{\alt\dpoint}		
\newmacro{\dpointaltalt}{\altalt\dpoint}		
\newmacro{\dpoints}{\mathcal{Y}}		

\newmacro{\dstate}{Y}		
\newmacro{\dbase}{v}		

\newcommand{\defeq}{\coloneqq}		

\newcommand{\from}{\colon}		

\newop{\Opt}{Opt}		
\newop{\Sol}{Sol}		
\newop{\gap}{Gap}		
\newop{\orcl}{Or}		

\newmacro{\tfun}{f}		
\newmacro{\obj}{f}		
\newmacro{\objalt}{g}		
\newmacro{\sobj}{F}		

\newmacro{\oper}{A}		
\newmacro{\vecfield}{g}		

\newcommand{\sol}[1][\point]{#1^{\ast}}		
\newmacro{\solvec}{\vecfield(\sol)}		
\newmacro{\solpay}{\eq[\payv]}		

\newmacro{\signal}{g}		
\newmacro{\step}{\gamma}		
\newmacro{\learn}{\eta}		

\newmacro{\vbound}{G}		
\newmacro{\lips}{L}		
\newmacro{\strong}{\mu}		
\newmacro{\smooth}{\beta}		

\newop{\cone}{cone}
\newop{\tspace}{T}		
\newop{\tcone}{TC}		
\newop{\dcone}{DC}		
\newop{\ncone}{NC}		
\newop{\pcone}{PC}		
\newop{\hull}{\Delta}		

\newmacro{\cvx}{\mathcal{C}}		
\newmacro{\subd}{\partial}		
\newcommand{\subsel}{\nabla}		

\newmacro{\minmax}{\mathcal{L}}		

\newmacro{\minvar}{\point_{1}}		
\newmacro{\minvaralt}{\alt\minvar}		
\newmacro{\minvars}{\points_{1}}		
\newmacro{\minsol}{\sol[\minvar]}		

\newmacro{\maxvar}{\point_{2}}		
\newmacro{\maxvaaltr}{\alt\maxvar}		
\newmacro{\maxvars}{\points_{2}}		
\newmacro{\maxsol}{\sol[\maxvar]}		

\newop{\Eucl}{\Pi}		
\newop{\logit}{\Lambda}		
\newop{\dkl}{KL}		

\newmacro{\hreg}{h}		
\newmacro{\hconj}{\hreg^{\ast}}		
\newmacro{\breg}{D}		
\newmacro{\mprox}{P}		
\newmacro{\mirror}{Q}		
\newmacro{\fench}{F}		
\newmacro{\hstr}{\revise{\kappa}}		
\newmacro{\depth}{H}		
\newmacro{\proxdom}{\points_{\hreg}}		
\newmacro{\zone}{\mathcal{D}}		
\newmacro{\hker}{\theta} 

\DeclarePairedDelimiterXPP{\proxof}[2]{\mprox_{#1}}{(}{)}{}{#2}		

\newmacro{\point}{x}		
\newmacro{\pointalt}{\alt\point}		
\newmacro{\pointaltalt}{\altalt\point}		
\newmacro{\points}{\mathcal{X}}		
\newmacro{\intpoints}{\relint\points}		

\newmacro{\base}{p}		
\newmacro{\basealt}{q}		
\newmacro{\basealtalt}{u}		

\newmacro{\open}{\mathcal{U}}		
\newmacro{\closed}{\mathcal{C}}		
\newmacro{\cpt}{\mathcal{K}}		
\newmacro{\nhd}{\mathcal{U}}		
\newmacro{\nhdalt}{\nhd}		

\newop{\ex}{\mathbb{E}}		
\newop{\prob}{\mathbb{P}}		
\newop{\Var}{Var}		
\newop{\simplex}{\hull}		

\DeclarePairedDelimiterXPP{\exof}[1]{\ex}{[}{]}{}{
 #1}

\DeclarePairedDelimiterXPP{\probof}[1]{\prob}{(}{)}{}{
 #1}

\DeclarePairedDelimiterXPP{\oneof}[1]{\one}{\{}{\}}{}{
 #1}

\newmacro{\sample}{\omega}		
\newmacro{\samples}{\Omega}		

\newmacro{\filter}{\mathcal{F}}		
\newmacro{\probspace}{(\samples,\filter,\prob)}		

\newmacro{\event}{E}       
\newmacro{\eventalt}{H}       
\newmacro{\mean}{\mu}		
\newmacro{\sdev}{\sigma}		
\newmacro{\variance}{\sdev^{2}}		

\newmacro{\proper}{\tau}		

\newmacro{\error}{Z}		
\newmacro{\noise}{U}		
\newmacro{\bias}{b}		
\newmacro{\brown}{W}		

\newmacro{\serror}{\theta}		
\newmacro{\snoise}{\xi}		
\newmacro{\sbias}{\psi}		

\newmacro{\sbound}{M}		
\newmacro{\bbound}{B}		
\newmacro{\noisepar}{\sdev}		
\newmacro{\noisevar}{\variance}		

\newmacro{\leg}{\alpha}		
\newcommand{\expleg}[1][{\sol[\legexp]}]{\frac{#1}{1-{#1}}}
\newmacro{\legexp}{\beta}		
\newmacro{\legsol}{\sol[\legexp]}		
\newmacro{\legconst}{\revise{K}}		
\newmacro{\legnhd}{\mathcal{U}}		
\DeclarePairedDelimiterXPP{\legof}[1]{\legexp_{\hreg}}{(}{)}{}{#1}
\newmacro{\bregexp}{\alpha} 
\newmacro{\bregcst}{M} 
\newmacro{\kernelcst}{C}
\newmacro{\kernelexp}{p}
\newmacro{\bregbdedcst}{R}

\newmacro{\constr}{j}		
\newmacro{\nConstr}{m}		

\newmacro{\slack}{\nu}		

\newmacro{\actcoords}{\mathcal{A}}
\newmacro{\flatcoords}{\actcoords_{\flat}}
\newmacro{\sharpcoords}{\actcoords_{\sharp}}

\newmacro{\drift}{\slack^{\ast}}
\newmacro{\drifteff}{\slack_{\mathrm{eff}}}
\newcommand{\sharps}{\sharpcoords}		
\newcommand{\flats}{\flatcoords}		

\newmacro{\coords}{\mathcal{I}}
\newmacro{\polycst}{P}

\newmacro{\intcst}{\sol[\coef]}
\newmacro{\intcstalt}{\widetilde{\coef}}

\newmacro{\expstep}{\eta}
\newmacro{\cst}{q} 
\newmacro{\cstalt}{q'} 
\newmacro{\seq}{u} 
\newmacro{\seqalt}{b}

\addauthor[Waïss]{WA}{DarkGreen}
\newcommand{\WA}{\WAmargincomment}
\newcommand{\half}[1][1]{{\frac{#1}{2}}}

\newcommand{\refinapp}[2]{\revise{\cref{#1} in \cref{#2}}}
\newcommand{\refapp}[2][in]{\revise{#1 \cref{#2}}}

\addauthor[Franck]{FI}{Purple}

\addauthor[Jérôme]{JM}{DarkRed}

\addauthor[Pan]{PM}{MediumBlue}
\newcommand{\PM}{\PMmargincomment}

\newmacro{\region}{\mathcal{R}}

\newmacro{\fn}{f}		
\newmacro{\fixmap}{F}		
\newmacro{\fixmapalt}{G}		
\newmacro{\gold}{\varphi}		
\newmacro{\silver}{\Phi}		
\newmacro{\energy}{E}		
\newmacro{\Lyap}{L}		
\newmacro{\diff}{\alpha} 

\newmacro{\thres}{\delta}		
\newmacro{\basin}{\mathcal{B}}		
\newmacro{\inhd}{\init[\nhd]}

\newmacro{\seed}{\theta}		
\newmacro{\seeds}{\Theta}		
\newmacro{\pdist}{P}		
\newmacro{\history}{\mathcal{H}}		

\begin{document}


\title
[The rate of convergence of Bregman proximal methods]
{The rate of convergence of Bregman proximal methods:\\
Local geometry vs. regularity vs. sharpness}

\author
[W.~Azizian]
{Waïss Azizian$^{c,\ast}$}
\address{$^{c}$\,%
Corresponding author.}
\address{$^{\ast}$\,%
Univ. Grenoble Alpes, CNRS, Inria, Grenoble INP, LJK, 38000 Grenoble, France.}
\email[Corresponding author]{waiss.azizian@univ-grenoble-alpes.fr}
\author
[F.~Iutzeler]
{Franck Iutzeler$^{\sharp}$}
\address{$^{\sharp}$\,%
Institut de Mathématiques de Toulouse,  Université de Toulouse,  CNRS, UPS, 31062, Toulouse, France.}
\email{franck.iutzeler@math.univ-toulouse.fr}
\author
[J.~Malick]
{\\Jérôme Malick$^{\ast}$}
\email{jerome.malick@cnrs.fr}
\author
[P.~Mertikopoulos]
{Panayotis Mertikopoulos$^{\diamond}$}
\address{$^{\diamond}$\,%
Univ. Grenoble Alpes, CNRS, Inria, Grenoble INP, LIG, 38000 Grenoble, France.}
\email{panayotis.mertikopoulos@imag.fr}
\thanks{The authors are grateful to J.~Bolte for many fruitful discussions.}



\newacro{LHS}{left-hand side}
\newacro{RHS}{right-hand side}
\newacro{iid}[i.i.d.]{independent and identically distributed}
\newacro{lsc}[l.s.c.]{lower semi-continuous}
\newacro{NE}{Nash equilibrium}
\newacroplural{NE}[NE]{Nash equilibria}

\newacro{ABP}{abstract Bregman proximal}
\newacro{BP}{Bregman proximal}
\newacro{BPM}{Bregman proximal method}

\newacro{DGF}{distance-generating function}
\newacro{EG}{extra-gradient}
\newacro{MP}{mirror-prox}
\newacro{MD}{mirror descent}
\newacro{OMD}{optimistic mirror descent}
\newacro{OMWU}{optimistic multiplicative weights update}
\newacro{PMP}{past mirror-prox}
\newacro{AMP}{abstract mirror-prox}
\newacro{MPT}{mirror-prox template}

\newacro{VI}{variational inequality}
\newacroplural{VI}[VIs]{variational inequalities}
\newacro{VIP}{variational inequality problem}
\newacro{KKT}{Karush\textendash Kuhn\textendash Tucker}
\newacro{FOS}{first-order stationary}
\newacro{SOO}{second-order optimality}
\newacro{SOS}{second-order sufficient}
\newacro{DGF}{distance-generating function}
\newacro{SFO}{stochastic first-order oracle}

\maketitle
\begin{abstract}
%
%
We examine the last-iterate convergence rate of \aclp{BPM} \textendash\ from \acl{MD} to \acl{MP} and its optimistic variants \textendash\ as a function of the local geometry induced by the prox-mapping defining the method.
For generality, we focus on local solutions of constrained, non-monotone \aclp{VI}, and we show that the convergence rate of a given method depends sharply on its associated \emph{Legendre exponent}, a notion that measures the growth rate of the underlying Bregman function (Euclidean, entropic, or other) near a solution.
In particular, we show that boundary solutions exhibit a stark separation of regimes between methods with a zero and non-zero Legendre exponent:
the former converge at a linear rate, while the latter converge, in general, sublinearly.
This dichotomy becomes even more pronounced in linearly constrained problems where methods with entropic regularization achieve a linear convergence rate along sharp directions, compared to convergence in a finite number of steps under Euclidean regularization.
\end{abstract}

\acresetall		

\section{Introduction}
\label{sec:introduction}

\Acp{BPM} have a long and rich history in optimization, going back at least to the introduction of \acl{MD} by Nemirovski \& Yudin \citep{NY83}.
In plain terms, \acp{BPM} are first-order (constrained) optimization algorithms that forego Euclidean projections in favor of a more sophisticated ``prox-mapping'' that minimizes a certain distance-like functional known as the Bregman divergence \citep{NY83,CT93,Bre67,Kiw97}.
When this Bregman divergence is the Euclidean distance squared, one recovers the standard projection-based methods;
other than that, depending on the problem's feasible region, different Bregman setups lead to a diverse collection of algorithms,
from exponentiated gradient descent on the simplex \citep{NY83,BecTeb03,ACBFS02},
to matrix multiplicative weights on the positive-semidefinite cone \cite{TRW05,KSST12},
variants of Karmarkar's affine scaling algorithm for linear programs \cite{VMF86},
etc.

One of the most appealing features of \acp{BPM} is that they achieve almost dimension-free convergence rates in problems with a convex structure and a favorable geometry \textendash\ such as the $L^{1}$ ball, the spectraplex, second-order cones, etc. \cite{Bub15,Nes09,BecTeb03}.
This is owed to a delicate interplay between the algorithms' non-Euclidean update scheme and the global geometry of the problem's domain.
However, these (almost) dimension-free guarantees also come with some strings attached:
they do not concern the sequence of iterates generated by the method, but only its time average
\revise{(or, through the same, ``regret-based'' analysis, the method's ``best iterate'')};
in this way, the best guarantee that can be achieved after $\run$ iterations is $\bigoh(1/\run)$.

In terms of oracle complexity, this is sufficient for problems that are not strongly convex\,/\,strongly monotone, but if one targets finer, geometric convergence rates,
\revise{the inherent averaging involved in regret-based guarantees is hard to compensate.}
And, on the other extreme, if the problem is not convex\,/\,monotone to begin with, iterate averaging does not provide any quantifiable benefits whatsoever, so it becomes crucial to study the actual trajectory of the method.

\para{Our contributions}

Our paper seeks to quantify the last-iterate convergence rate of \aclp{BPM} as a function of the Bregman divergence defining the method and the local geometry that it induces.
To treat this question in as general a manner as possible, we focus on \ac{VI} problems of the form
\begin{equation}
\label{eq:VI}
\tag{VI}
\text{Find $\sol\in\points$ such that}
	\;\;
	\braket{\vecfield(\sol)}{\point - \sol}
	\geq 0
	\;\;
	\text{for all $\point\in\points$},
\end{equation}
where $\points$ is a closed convex subset of a finite-dimensional normed space $\pspace$, and $\vecfield \from \points \to \dspace$ is a (possibly non-monotone) single-valued operator on $\points$ with values in $\dspace$, the dual of $\pspace$.
This problem is a staple of many areas of mathematical programming, game theory and data science, as it provides a template for ``optimization beyond minimization'' \textendash\ \ie for problems where finding an optimal solution does not necessarily involve minimizing a loss function.
In particular, in addition to standard minimization problems \textendash\ which are recovered when $\vecfield = \nabla\obj$ for some smooth function $\obj$ \textendash\ the general formulation \eqref{eq:VI} includes saddle-point problems, games, complementarity problems, etc.;
for an introduction, see \cite{FP03} and references therein.

In this broad context, we examine the rate of convergence of a wide class of \aclp{BPM} to local solutions of \eqref{eq:VI} that satisfy a \acl{SOS} condition.
Specifically, the class of algorithms we consider includes as special cases
\begin{enumerate*}
[(\itshape i\hspace*{1pt}\upshape)]
\item
the original \acf{MD} algorithm of \cite{NY83};
\item
the \acf{MP} method of Nemirovski \cite{Nem04} \textendash\ which has the same update structure as the Bregman-based algorithm of \cite{AT05} and contains as a special case the \acf{EG} algorithm of \cite{Kor76};
\item
the so-called \acf{OMD} method of \cite{RS13-NIPS} \textendash\ itself a Bregman analogue of the modified Arrow-Hurwicz algorithm of \cite{Pop80};
\end{enumerate*}
etc.

Our first finding is a crisp characterization of last-iterate convergence rate of \acp{BPM} in terms of the local geometry induced by the underlying Bregman function near a given solution of \eqref{eq:VI}.
We make this dependence precise via the notion of the \emph{Legendre exponent}, a regularity measure for Bregman methods due to \cite{AIMM21}, which can roughly be described as the logarithmic ratio of the volume of a Euclidean ball to that of a Bregman ball of the same radius.
For example, Euclidean methods have a Legendre exponent of $\legexp = 0$ and they converge at a linear rate;
entropic methods have a Legendre exponent of $\legexp = 1/2$ at boundary points, and they converge at a rate of $\bigoh(\run^{-1})$;
more generally,
as we show in \cref{thm:general}, methods with a Legendre exponent $\legexp>0$ converge at a rate of $\bigoh(\run^{1-1/\legexp})$.
\PM{We need to fix this: the $1-1/\legexp$ exponent is not consistent with the $\bigoh(1/\run)$ expression.}
\WA{I don't see the issues, yes this expression is not well-defined for $\legexp = 0$ but this is normal, the two situations differ radically.}
The Euclidean regime ($\legexp = 0$) is perfectly aligned with existing results for the geometric last-iterate convergence rate of the \ac{EG} algorithm and its variants \citep{GBVV+19,Mal15,HIMM19,MOP20}.
By contrast, the Legendre regime ($\legexp > 0$) indicates a significant drop in the algorithm's last-iterate convergence speed, even though ergodic convergence rates \cite{Nes04} and results for bilinear games \cite{WLZL21} might suggest otherwise.

Subsequently, motivated by applications to game theory and linear programming, we take a closer look at the convergence rate of \acp{BPM} across the constraints that are active at a solution $\sol$ of \eqref{eq:VI} depending on the position of $\vecfield(\sol)$ relative to said constraints. 
This analysis reveals that Bregman proximal methods have a particularly fine structure:
along \emph{sharp directions} (\ie constraints along which $\vecfield(\sol)$ is strictly inward-pointing), \acp{BPM} converge
\begin{enumerate*}
[(\itshape i\hspace*{1pt}\upshape)]
\item
at a rate of $\bigoh(1/\run^{1/(2\legexp-1)})$ if $1/2 < \legexp < 1$;
\item
at a \emph{geometric rate} if $0 < \legexp \leq 1/2$ (\eg for entropic methods);
and
\item
in a \emph{finite} number of iterations if $\legexp=0$
\end{enumerate*}
(\cf \cref{thm:sharp}).
Thus, even though the estimates of \cref{thm:general} are, in general tight, the actual convergence rate of a Bregman method along different coordinates\,/\,constraints could be starkly different \textendash\ and, in fact, dramatically faster if the solution under study is itself sharp.

The closest antecedent of our work is the conference paper \cite{AIMM21} where the Legendre exponent was introduced to analyze the convergence of \ac{OMD} in \emph{stochastic} \ac{VI} problems (without considering sharp directions and/or faster identification rates).
The stochastic and deterministic settings are obviously very different, both in the challenges involved as well as the rates obtained, so there is no overlap in our analysis and results.
Other than that, we are not aware of any comparable results in the literature concerning the radically different convergence landscape of \acp{BPM} along active and inactive constraints.

\section{Problem setup and preliminaries}
\label{sec:setup}

In the rest of our paper
$\pspace$ will denote an $\nCoords$-dimensional real space with norm $\norm{\cdot}$
and
$\points$ will be a closed convex subset thereof.
We will also write
$\dpoints \defeq \dspace$ for the dual of $\pspace$,
$\braket{\dpoint}{\point}$ for the canonical pairing between $\dpoint\in\dpoints$ and $\point\in\pspace$,
and
$\dnorm{\dpoint} \defeq \max \setdef{\braket{\dpoint}{\point}}{\norm{\point}\leq 1}$ for the induced dual norm on $\dpoints$.

\subsection{Blanket assumptions}

Throughout the sequel, we will make the following assumptions for the defining vector field $\vecfield\from\points\to\dpoints$ of \eqref{eq:VI} and the solution $\sol\in\points$ under study:

\begin{assumption}
[Lipschitz continuity]
\label{asm:Lipschitz}
The vector field $\vecfield$ is \emph{$\lips$-Lipschitz continuous}, \ie
\begin{equation}
\label{eq:Lipschitz}
\tag{LC}
\dnorm{\vecfield(\pointalt) - \vecfield(\point)}
	\leq \lips \norm{\pointalt - \point}
	\quad
	\text{for all $\point,\pointalt\in\points$}.
\end{equation}
\end{assumption}

\begin{assumption}
[Second-order sufficiency]
\label{asm:strong}
There exists
a convex neighborhood $\basin$ of $\sol$ in $\points$
and
a positive constant $\strong > 0$
such that
\begin{equation}
\label{eq:strong}
\tag{SOS}
\braket{\vecfield(\point) - \vecfield(\sol)}{\point - \sol}
	\geq \strong \norm{\point - \sol}^{2}
	\quad
	\text{for all $\point\in\basin$}.
\end{equation}
\end{assumption}

In general, \Cref{asm:strong} guarantees that $\sol$ is the unique solution of \eqref{eq:VI} in $\basin$;
we illustrate this in two special cases of interest:
\begin{enumerate}
\item
\emph{Minimization problems:}
suppose that $\vecfield = \nabla\obj$ for some Lipschitz smooth objective function $\obj$ on $\points$.
Then, \cref{asm:strong} implies that $\obj$ grows (at least) quadratically along every ray emanating from $\sol$, \ie $\obj(\point) - \obj(\sol) \geq \braket{\subsel\obj(\sol)}{\point - \sol} + (\strong/2) \norm{\point - \sol}^{2} = \Omega(\norm{\point-\sol}^{2})$ for all $\point\in\basin$, implying in particular that $\sol$ is an isolated minimizer of $\obj$.
\item
\emph{Min-max problems:}
suppose that $\points$ factorizes as $\points = \minvars\times\maxvars$ for suitable factor sets $\minvars$, and $\maxvars$,
let $\minmax\from\points\to\R$ be a smooth function on $\points$,
and
write $\vecfield = (\nabla_{\minvar}\minmax,-\nabla_{\maxvar}\minmax)$ for the min-max gradient of $\minmax$ (with respect to $\minvar\in\minvars$ and $\maxvar\in\maxvars$ respectively).
Then, any solution $\sol = (\minsol,\maxsol)$ of \eqref{eq:VI} that satisfies \cref{asm:strong} enjoys the local growth bounds $\minmax(\minvar,\maxsol) - \minmax(\minsol,\maxsol) = \Omega(\norm{\minvar - \minsol}^{2})$ and $\minmax(\minsol,\maxsol) - \minmax(\minsol,\maxvar) = \Omega(\norm{\maxvar - \maxsol}^{2})$, implying in turn that $\sol$ is an isolated, hyperbolic saddle-point of $\minmax$.
\end{enumerate}
More examples satisfying \eqref{eq:strong} include strict \aclp{NE} in finite games \cite{FT91}, deterministic Nash policies in (generic) stochastic games \cite{SV15}, etc.
Overall, \cref{asm:Lipschitz,asm:strong} apply to a very wide range of problems, so we will treat them as blanket assumptions throughout.

\subsection{\acl{BP} methods}
\label{subsec:BP_methods}

As we discussed in the introduction, the main algorithmic template that we will examine for solving \eqref{eq:VI} is a general class of first-order algorithms known as \acdefp{BPM}.
The defining ingredient of this class is the notion of \emph{Bregman regularizer}, which we define below as follows:

\begin{definition}
[Bregman regularizers and related notions]
\label{def:Bregman}
A proper, \acl{lsc}, strictly convex function $\hreg\from\pspace\to\R\cup\{\infty\}$ is a \emph{Bregman regularizer} on $\points$ if
the following are true
\begin{enumerate}
\item
$\hreg$ is supported on $\points$, \ie $\dom\hreg = \points$.
\item
The subdifferential of $\hreg$ admits a \emph{continuous selection}, \ie there exists a continuous mapping $\subsel\hreg\from\dom\subd\hreg\to\dpoints$ such that $\subsel\hreg(\point) \in \subd\hreg(\point)$ for all $\point\in\dom\subd\hreg$.
\item
$\hreg$ is \emph{locally strongly convex} relative to $\norm{\cdot}$, \ie for any compact set $\cpt \subseteq \dom \hreg$, we have
\begin{equation}
\label{eq:hstr}
\hreg(\pointalt)
	\geq \hreg(\point)
		+ \braket{\subsel\hreg(\point)}{\pointalt - \point}
		+ \tfrac{\hstr}{2} \norm{\pointalt - \point}^{2}
\end{equation}
for some $\hstr>0$ and for all $\point\in \cpt \cap \dom\subd\hreg$, $\pointalt\in \cpt$. 
\end{enumerate}
\noindent
The set $\proxdom \defeq \dom\subd\hreg$ will be referred to as the \emph{prox-domain} of $\hreg$.
In addition, we also define the \emph{Bregman divergence} of $\hreg$ as
\begin{alignat}{2}
\label{eq:Breg}
\breg(\base,\point)
	&= \hreg(\base)
		- \hreg(\point)
		- \braket{\subsel\hreg(\point)}{\base - \point}
	&\qquad
	&\text{for all $\point\in\proxdom$, $\base\in\points$}
\intertext{and the induced \emph{prox-mapping} as}
\label{eq:prox}
\proxof{\point}{\dvec}
	&= \argmin\nolimits_{\pointalt\in\points} \{ \braket{\dvec}{\point - \pointalt} + \breg(\pointalt,\point) \}
	&\qquad
	&\text{for all $\point\in\proxdom$, $\dvec\in\dpoints$}.
\end{alignat}
\end{definition}

\begin{remark}
\label{rem:Bregman}
Examples of Bregman regularizers are given in \cref{sec:examples}, where we also take an in-depth look at their properties.
For our analysis and results, we will assume for convenience that $\hreg$ is $1$-strongly convex in a suitable neighborhood $\zone$ of $\sol$ which will be understood from the context;
this can always be achieved by rescaling $\hreg$, so there is no loss of generality.
To avoid technicalities, we will also tacitly assume that $\proxof{\point}{\dpoint}$ is well-defined whenever it is invoked (this is always the case if, for example, $\hreg$ is coercive or $\nabla\hreg$ is invertible).
\end{remark}

Given a Bregman regularizer on $\points$, the general class of \acdefp{BPM} that we will consider is defined via the generic recursion
\begin{equation}
\label{eq:BPM}
\tag{BPM}
\begin{aligned}
\lead
	= \proxof{\curr}{-\curr[\step]\curr[\signal]}
	\qquad
\next
	= \proxof{\curr}{-\curr[\step]\lead[\signal]}
\end{aligned}
\end{equation}
where
\begin{enumerate*}
[(\itshape i\hspace*{1pt}\upshape)]
\item
$\run=\running$ denotes the method's iteration counter;
\item
$\curr[\step] > 0$ is a (non-increasing) step-size sequence;
\item
$\curr[\signal]$ and $\lead[\signal]$ are sequences of ``oracle signals'' that we discuss in detail below.
\end{enumerate*}
In terms of vocabulary, the iterates $\curr$, $\run=\running$, will be referred to as the ``\emph{base states}'' of the method, while the ``half-iterates'' $\lead$, $\run=\running$, will be referred to as the method's ``\emph{leading states}''.
Finally, in terms of initialization, we will take for convenience $\init = \state_{1/2}$.

Now, regarding the sequence of oracle signals $\curr[\signal]$ and $\lead[\signal]$ defining \eqref{eq:BPM}, we will assume throughout that
\begin{equation}
\label{eq:signal-lead}
\lead[\signal]
	= \vecfield(\lead)
	\qquad
	\text{for all $\run=\running$}
\end{equation}
\ie \eqref{eq:BPM} generates a new base state $\next$ by taking a Bregman proximal step from $\curr$ with oracle input from the leading state $\lead$.
By contrast, the leading state itself can be generated in a number of different ways from $\curr$, depending on the definition of $\curr[\signal]$:

\begin{assumption}
\label{asm:signal-base}
For all $\run=\running$, the oracle signal $\curr[\signal]$ is of the form:
\begin{equation}
\label{eq:signal-base}
\curr[\signal]
	= \coef[a] \vecfield(\curr)
		+ \coef[b] \vecfield(\beforelead)
\end{equation}
for some $\coef[a],\coef[b] \in [0,1]$ with
$\coef[a] + \coef[b] \leq 1$
and
$\coef[a] + \coef[b] = 1$ if $\coef[b]>0$.\footnote{Note that the requirement ``$\coef[a] + \coef[b] = 1$ if $\coef[b] > 0$'' is only intended to ease notation and does not lead to a loss in generality:
if $\coef[b] > 0$, we can always rescale $\curr[\step]$ by $\coef[a] + \coef[b]$ so the condition $\coef[a] + \coef[b] = 1$ is satisfied automatically.}
\end{assumption}

For concreteness, we illustrate below three archetypal Bregman methods that serve as the backbone of the above framework:

\begin{enumerate}
\setlength{\itemsep}{0pt}
\item
\Acli{MD}:
following \cite{NY83,BecTeb03,NJLS09}, the \acli{MD} algorithm proceeds recursively as $\new = \proxof{\point}{-\step\vecfield(\point)}$, so it can be recovered from \eqref{eq:BPM} by taking
\begin{alignat}{3}
\label{eq:MD}
\tag{MD}
\coef[a] = 0, \coef[b] = 0
	&\qquad
	\text{or, equivalently}
	&\qquad
\curr[\signal]
	&= 0
	&\quad
	\text{for all $\run=\running$}
\intertext{%
\item
\Acli{MP}:
following \cite{Nem04,JNT11}, the \acli{MP} algorithm corresponds to the choice}
\label{eq:MP}
\tag{MP}
\coef[a] = 1, \coef[b] = 0
	&\qquad
	\text{or, equivalently}
	&\qquad
\curr[\signal]
	&= \vecfield(\curr)
	&\quad
	\text{for all $\run=\running$}
\intertext{%
\item
\Acli{OMD}:
originally due to \cite{Pop80} (in the Euclidean case) and \cite{CYLM+12,RS13-NIPS} (for the general case), the \acli{OMD} algorithm is obtained by setting}
\label{eq:OMD}
\tag{OMD}
\coef[a] = 0, \coef[b] = 1
	&\qquad
	\text{or, equivalently}
	&\qquad
\curr[\signal]
	&= \vecfield(\beforelead)
	&\quad
	\text{for all $\run=\running$}
\end{alignat}
\end{enumerate}
These three algorithms are the most widely studied Bregman methods in the literature, so we will use them as running examples throughout.

\section{Motivating examples}
\label{sec:examples}

We now proceed to take a closer look at some commonly used Bregman regularizers (and the induced prox-mappings) with the goal of determining the rate of convergence of the associated Bregman method.
For concreteness, we focus on one-dimensional problems where $\points$ is \revise{the closed interval $[0,\infty)$ or $[-1, 1]$} and $\vecfield$ is the affine vector field
\begin{equation}
\label{eq:simple}
\vecfield(\point)
	= \point - \sol,
	\quad
	\point\in\R,
\end{equation}
for different choices of $\sol\in\R$
\revise{(typically a boundary point of $\points$)}.
To streamline our presentation, we will only examine the \acl{MD} recursion \eqref{eq:MD} with constant step-size schedules $\curr[\step] \equiv \step$ for some $\step>0$.
In this case, we obtain the scheme
\begin{equation}
\label{eq:MD-generic}
\next
	= \fixmap(\curr)
	\quad
	\text{with}
	\quad
\fixmap(\point)
	= \proxof{\point}{-\step\vecfield(\point)},
\end{equation}
and we will examine the convergence speed of $\curr$ by analyzing the behavior of $\fixmap$ near $\sol$.
\revise{To illustrate the spectrum of different behaviors that arise near the boundary of $\points$, we will focus primarily on cases where $\sol$ is a boundary point.}
\smallskip


\begin{example}
[Euclidean regularization]
\label{ex:Eucl}
We begin with the quadratic regularizer $\hreg(\point) = \point^{2}/2$ for $\point\in\points = [0,\infty)$.
In this case, noting that $\hreg'(\point) = \point$, we have:
\begin{flalign}
\label{eq:mirror-Eucl}
\begin{alignedat}{3}
\quad
	a)\;\;
	&\text{Prox-domain:}
	&\qquad
&\proxdom
	= \points
	&\\
\quad
	b)\;\;
	&\text{Bregman divergence:}
	&\qquad
&\breg(\base,\point)
	= (\base-\point)^{2}/2
	&\\
\quad
	c)\;\;
	&\text{Prox-mapping:}
	&\qquad
&\proxof{\point}{\dvec}
	= \pospart{\point+\dvec}
	&
\end{alignedat}
&&
\end{flalign}
Consider now the case $\sol=0$, \ie $\vecfield(\point) = \point$.
Then, for $\step\in(0,1)$, the update \eqref{eq:MD-generic} becomes
\begin{equation}
\label{eq:MD-Eucl}
\fixmap(\point)
	= \point - \step\point
	= (1-\step) \point
    \quad
    \text{\revise{for all $\point\geq0$}}
\end{equation}
\ie $\fixmap$ is contracting.
We thus conclude that $\curr$ converges to $\sol=0$ at a geometric rate, \viz
\begin{equation*}
\tag*{\endenv}
\breg(\sol,\curr)
	= \tfrac{1}{2} \curr^{2}
	= \Theta\parens[\big]{(1-\step)^{2\run}}
	\;\;
	\text{or, in absolute value,}
	\;\;
\abs{\curr - \sol}
	= \Theta\parens{(1-\step)^{\run}}.
\end{equation*}
\end{example}



\begin{example}
[Entropic regularization]
\label{ex:ent}
Another popular choice when $\points=[0,\infty)$ is the entropic regularizer $\hreg(\point) = \point\log\point$ \cite{BecTeb03,SS11,BBT17}.
In this case, we have $\hreg'(\point) = 1 + \log\point$, and hence:
\begin{flalign}
\label{eq:mirror-ent}
\begin{alignedat}{3}
\quad
	a)\;\;
	&\text{Prox-domain:}
	&\qquad
&\proxdom
	= \relint\points
	= (0,\infty)
	&\\
\quad
	b)\;\;
	&\text{Bregman divergence:}
	&\qquad
&\breg(\base,\point)
	= \base \log(\base/\point) + \point - \base
	&\\
\quad
	c)\;\;
	&\text{Prox-mapping:}
	&\qquad
&\proxof{\point}{\dvec}
	= \point \exp(\dvec).
	&
\end{alignedat}
&&
\end{flalign}
Now, taking $\vecfield(\point) = \point$ as in the previous example, the update rule \eqref{eq:MD-generic} becomes
\begin{equation}
\label{eq:MD-ent}
\fixmap(\point)
	= \point \exp(-\step\point)
	= \point(1 - \step\point  + o(\point))
	= \point - \step\point^2  + o(\point^2)
	\quad
	\text{as $\point\to0$}.
\end{equation}
In contrast to \eqref{eq:MD-Eucl}, we now have $\fixmap(\point) \sim \point$ instead of $(1-\step)\point$, so $\fixmap$ is no longer a contraction.
Instead, the iterates of \eqref{eq:MD-ent} may be analyzed by means of the following lemma:
\begin{restatable}{lemma}{basicnum}
\label{lem:basicnum}
Suppose that $\fn\from\R_+\to\R_+$ admits the asymptotic expansion
\begin{equation}
\fn(\point)
	= \point
		- \coef\point^{1+\rexp}
		+ o(\point^{1+\rexp})
	\quad
	\text{as $\point\to0$}
\end{equation}
for positive constants $\coef,\rexp>0$.
Then, for $\init[\seq] > 0$ small enough, the sequence $\next[\seq] = \fn(\curr[\seq])$, $\run=\running$, converges to $0$ at a rate of $\curr[\seq] \sim (\coef\rexp\run)^{-1/\rexp}$.
\end{restatable}

Thanks to this lemma (which we prove \refapp{app:aux}), we readily conclude that $\curr$ converges to $0$ at a rate of
$\breg(\sol,\curr)
	= \curr
	= \abs{\curr - \sol}
	\sim 1/(\step\run).$
\hfill
\endenv
\end{example}



\begin{example}
[Fractional power]
\label{ex:frac}
Take $\points = [0,\infty)$ and $\vecfield(\point) = \point$ as in \cref{ex:Eucl,ex:ent} above.
Then, for a given $\qexp>0$, $\qexp\neq1$, the \emph{fractional power} regularizer \textendash\ or \emph{Tsallis entropy} \textendash\ on $\points$ is defined as $\hreg(\point) = [\qexp(1-\qexp)]^{-1} (\point - \point^{\qexp})$ \citep{Tsa88,ABB04,MS16}.
For this choice of regularizer, we have $\hreg'(\point) = (1 - \qexp\point^{\qexp-1}) / [\qexp(1-\qexp)]$, and a series of direct calculations gives:%
\footnote{Strictly speaking, the expression we provide for $\proxof{\point}{\dvec}$ is only valid when $\dvec < \point^{\qexp-1}/(1-\qexp)$.
\revise{The reason for this is that the} prox-mapping $\proxof{\point}{\dvec}$ is not well-defined for all values of $\dvec$;
this detail is not important in the calculations that follow, so we disregard it for now.}
\begin{flalign}
\label{eq:mirror-frac}
\begin{alignedat}{3}
\quad
	a)\;\;
	&\text{Prox-domain:}
	&\qquad
&\text{$\proxdom = (0,\infty)$ if $\qexp\in(0,1)$ and $\proxdom = [0,\infty)$ if $\qexp > 1$}
	&\\
\quad
	b)\;\;
	&\text{Bregman divergence:}
	&\qquad
&\breg(\base,\point)
	= \frac{\point^{\qexp} - \base^{\qexp}}{\qexp(1-\qexp)}
		- \point^{\qexp-1} \frac{\point - \base}{1-\qexp}
	&\\
\quad
	c)\;\;
	&\text{Prox-mapping:}
	&\qquad
&\proxof{\point}{\dvec}
	= \bracks[\big]{\point^{\qexp-1} - (1-\qexp) \dvec}^{\frac{1}{\qexp-1}}
	\quad
	\text{for $\qexp\in(0,1)$}.
	&
\end{alignedat}
&&
\end{flalign}
Now, when applied to $\vecfield(\point) = \point$, the fractional power variant of \eqref{eq:MD-generic} for $\qexp\in(0,1)$ gives
\begin{equation}
\label{eq:MD-frac}
\fixmap(\point)
	= \point \, \bracks{1 + \step(1-\qexp)\point^{2-\qexp}}^{1/(\qexp-1)}
	= \point - \step\point^{3-\qexp} + o(\point^{3-\qexp})
	\quad
	\text{as $\point\to0$}.
\end{equation}
Hence, by \cref{lem:basicnum}, we conclude that $\curr$ converges to $0$ at a rate of
\begin{equation}
\tag*{\endenv}
\breg(\sol,\curr)
	= \Theta\parens[\big]{\run^{-\qexp/(2-\qexp)}}
	\;\;
	\text{or, in absolute value,}
	\;\;
\abs{\curr - \sol}
	= \Theta\parens[\big]{\run^{-1/(2-\qexp)}}.
\end{equation}
\end{example}



\begin{example}
[Hellinger distance]
\label{ex:Hell}
Our last example concerns the Hellinger regularizer $\hreg(\point) = -\sqrt{1-\point^{2}}$ on $\points = [-1,1]$.
Since $\hreg'(\point) = \point / \sqrt{1-\point^{2}}$, we readily obtain the following:
\begin{flalign}
\label{eq:mirror-Hell}
\begin{alignedat}{3}
\quad
	a)\;\;
	&\text{Prox-domain:}
	&\qquad
&\proxdom = \relint\points = (-1,1)
	&\\
\quad
	b)\;\;
	&\text{Bregman divergence:}
	&\qquad
&\breg(\base,\point)
	= \frac{1 - \base\point - \sqrt{(1-\base^{2})(1-\point^{2})}}{\sqrt{1-\point^{2}}}
	&\\
\quad
	c)\;\;
	&\text{Prox-mapping:}
	&\qquad
&\proxof{\point}{\dvec}
	= \frac{\point + \dvec\sqrt{1-\point^{2}}}{\sqrt{1-\point^{2} + (\point + \dvec\sqrt{1-\point^{2}})^{2}}}.
	&
\end{alignedat}
&&
\end{flalign}
In this case, taking $\vecfield(\point) = \point$ as per the previous examples, yields
\begin{equation}
\fixmap(\point)
	= \frac{\point - \step\point\sqrt{1-\point^{2}}}{\sqrt{1-\point^{2} + (\point - \step\point\sqrt{1-\point^{2}})^{2}}}
	\sim \point - \step\point
	\quad
	\text{as $\point\to0$},
\end{equation}
\ie $\curr$ converges to $\sol = 0$ at a geometric rate, as in \cref{ex:Eucl}.
On the other hand, if we consider the shifted operator $\vecfield(\point) = \point+1$, a somewhat tedious calculation (which we detail \revise{\refapp{app:ex}}) gives the following Taylor expansion near $\sol = -1$:
\begin{equation}
\fixmap(\point)
	= \sol
		+ (\point - \sol)
		- 2\sqrt{2}\step (\point - \sol)^{5/2} + o\parens*{(\point-\sol)^{5/2}}.
\end{equation}
Hence, by \cref{lem:basicnum}, we conclude that $\curr$ converges to $\sol = -1$ at a rate of
\begin{equation}
\tag*{\endenv}
\breg(\sol,\curr)
	= \Theta\parens{\run^{-1/3}}
	\;\;
	\text{or, in absolute value,}
	\;\;
\abs{\curr - \sol}
	= \Theta(\run^{-2/3}).
\end{equation}
\end{example}



\begin{figure}
\centering
\begin{subfigure}[b]{.45\linewidth}
\resizebox{\textwidth}{!}{
\begin{tikzpicture}

\definecolor{color0}{rgb}{1,1,0}

\begin{axis}[
legend pos=south east,
legend cell align={left},
legend style={fill opacity=0.8, draw opacity=1, text opacity=1, draw=white!80!black},
log basis y={10},
tick align=outside,
tick pos=left,
grid=both,
grid style={line width=.1pt, draw=gray!20},
major grid style={line width=.2pt,draw=gray!60},
xlabel={Iterations ($\run$)},
xmin=0, xmax=100,
xtick style={color=black},
ylabel={$\abs{\curr - \sol}$},
ymin=0.0005, ymax=1,
ymode=log,
ytick style={color=black}
]
\addplot [thick, black]
table {%
0 0.799999952316284
1 0.399999976158142
2 0.200000047683716
3 0.100000023841858
4 0.0499999523162842
5 0.0249999761581421
6 0.0125000476837158
7 0.00625002384185791
8 0.00312495231628418
9 0.00156247615814209
10 0.00078129768371582
12 0.00019526481628418
17 6.07967376708984e-06
100 0
};
\addlegendentry{Euclidean}
\addplot [thick, blue]
table {%
0 0.799999952316284
1 0.536256074905396
2 0.410133838653564
3 0.334092140197754
4 0.282695531845093
5 0.245432734489441
6 0.21708881855011
7 0.194758892059326
8 0.176687479019165
9 0.161747932434082
10 0.149181723594666
11 0.138458967208862
12 0.129197835922241
13 0.121115684509277
14 0.113998770713806
15 0.107682704925537
16 0.102038145065308
17 0.0969629287719727
18 0.0923740863800049
19 0.0882046222686768
20 0.0843991041183472
22 0.077703595161438
24 0.0720008611679077
26 0.0670843124389648
28 0.0628011226654053
30 0.0590356588363647
33 0.0541694164276123
36 0.0500485897064209
39 0.0465136766433716
43 0.0425139665603638
47 0.039150595664978
52 0.0356301069259644
57 0.0326927900314331
63 0.0297517776489258
70 0.0269277095794678
78 0.0242940187454224
88 0.0216491222381592
100 0.0191489458084106
};
\addlegendentry{Entropy}
\addplot [thick, red]
table {%
0 0.799999952316284
1 0.557325839996338
2 0.441132783889771
3 0.370003581047058
4 0.321104645729065
5 0.285059571266174
6 0.257208108901978
7 0.234940648078918
8 0.216669678688049
9 0.20136833190918
10 0.188340544700623
11 0.177095890045166
12 0.167278051376343
13 0.158621788024902
14 0.150924682617188
15 0.144029855728149
16 0.137813329696655
17 0.132176041603088
18 0.127037644386292
19 0.122332215309143
20 0.118005275726318
22 0.110311508178711
24 0.1036696434021
26 0.097870945930481
28 0.0927591323852539
30 0.0882148742675781
32 0.0841456651687622
35 0.0787761211395264
38 0.0741243362426758
41 0.0700509548187256
45 0.0653417110443115
49 0.061292290687561
54 0.0569581985473633
59 0.0532597303390503
65 0.0494743585586548
72 0.0457538366317749
80 0.0422004461288452
89 0.0388740301132202
99 0.0358034372329712
100 0.035525918006897
};
\addlegendentry{Fractional Power ($\qexp=3/4$)}
\addplot [thick, color0!50!black]
table {%
0 0.799999952316284
1 0.682318210601807
2 0.509535908699036
3 0.319698929786682
4 0.174822449684143
5 0.0897815227508545
6 0.0452085733413696
7 0.0226447582244873
8 0.0113275051116943
9 0.00566434860229492
10 0.00283229351043701
11 0.00141608715057373
12 0.000708103179931641
14 0.00017702579498291
19 5.48362731933594e-06
100 0
};
\addlegendentry{Hellinger}
\addplot [thick, green!50.1960784313725!black]
table {%
0 0.799999952316284
1 0.482910990715027
2 0.354313373565674
3 0.284990310668945
4 0.24113941192627
5 0.210627675056458
6 0.188022255897522
7 0.170514702796936
8 0.156500339508057
9 0.144992589950562
10 0.13534939289093
11 0.127134203910828
12 0.120038866996765
13 0.113839387893677
14 0.108369112014771
15 0.103500604629517
16 0.0991353988647461
17 0.0951956510543823
18 0.0916192531585693
20 0.085363507270813
22 0.0800628662109375
24 0.0755046606063843
26 0.0715363025665283
28 0.0680451393127441
31 0.0635222196578979
34 0.0596752166748047
37 0.0563565492630005
41 0.0525728464126587
45 0.0493607521057129
50 0.0459606647491455
56 0.0425640344619751
63 0.0392997264862061
71 0.0362429618835449
80 0.0334290266036987
91 0.0306370258331299
100 0.0287426710128784
};
\addlegendentry{Hellinger (shifted)}
\end{axis}

\end{tikzpicture}}
\end{subfigure}
\hfill
\begin{subfigure}[b]{.45\linewidth}
\resizebox{\textwidth}{!}{
\begin{tikzpicture}

\begin{axis}[
legend pos=south west,
legend cell align={left},
legend style={fill opacity=0.8, draw opacity=1, text opacity=1, draw=white!80!black},
log basis x={10},
log basis y={10},
tick align=outside,
tick pos=left,
grid=both,
grid style={line width=.1pt, draw=gray!20},
major grid style={line width=.2pt,draw=gray!60},
xlabel={Iterations ($\run$)},
xmin=1, xmax=10000,
xmode=log,
xtick style={color=black},
ylabel={$\abs{\curr - \sol}$},
ymin=0.0005, ymax=1,
ymode=log,
ytick style={color=black}
]
\addplot [thick, blue]
table {%
1 0.536255955696106
2.00000023841858 0.410133868455887
2.99999976158142 0.334092080593109
4.00000095367432 0.282695561647415
5.00000095367432 0.245432749390602
6.00000047683716 0.217088833451271
6.99999952316284 0.19475881755352
8.99999904632568 0.161747917532921
11 0.138458997011185
14.0000009536743 0.113998793065548
18 0.0923740938305855
24 0.0720009058713913
31.9999980926514 0.0556991770863533
43.9999961853027 0.0416197888553143
62.9999923706055 0.0297517590224743
95.9999923706055 0.0199154745787382
159.000015258789 0.0122197614982724
300.000061035156 0.00655778544023633
715.999938964844 0.00277234171517193
2621.00024414062 0.000761307892389596
10000 0.000199865506147034
};
\addlegendentry{Entropy}
\addplot [thick, blue!50!black, dashed]
table {%
5.00000095367432 0.399691134691238
9999.0009765625 0.000199865506147034
};
\addlegendentry{theory: $\Theta(\run^{-1})$}
\addplot [thick, red]
table {%
1 0.557325899600983
2.00000023841858 0.441132843494415
2.99999976158142 0.370003491640091
4.00000095367432 0.32110458612442
5.00000095367432 0.28505951166153
6.00000047683716 0.257208108901978
8 0.216669619083405
10 0.188340544700623
13 0.158621773123741
17.0000019073486 0.132176086306572
23 0.106873579323292
31.9999980926514 0.0841456726193428
47.0000076293945 0.0632438734173775
73 0.0452729798853397
122.999977111816 0.0302517041563988
235.999954223633 0.0181559771299362
578 0.00893514323979616
2390.00024414062 0.00288305035792291
10000 0.000918549543712288
};
\addlegendentry{Fractional Power ($\qexp=3/4$)}
\addplot [thick, red!50!black, dashed]
table {%
5.00000095367432 0.401690781116486
9999.0009765625 0.000918549543712288
};
\addlegendentry{theory: $\Theta(\run^{-4/5})$}
\addplot [thick, green!50.1960784313725!black]
table {%
1 0.482910990715027
2.00000023841858 0.354313313961029
2.99999976158142 0.284990191459656
4.00000095367432 0.24113941192627
6.00000047683716 0.188022255897522
8.99999904632568 0.144992560148239
14.0000009536743 0.108369030058384
26.0000019073486 0.0715362876653671
1499.00036621094 0.00464871991425753
10000 0.00130712543614209
};
\addlegendentry{Hellinger (shifted)}
\addplot [thick, green!40!black, dashed]
table {%
5.00000095367432 0.20747934281826
9999.0009765625 0.00130712543614209
};
\addlegendentry{theory: $\Theta(\run^{-2/3})$}
\end{axis}

\end{tikzpicture}}
\end{subfigure}
\caption{The rate of convergence of \eqref{eq:MD} in \crefrange{ex:Eucl}{ex:Hell}.
The Euclidean and shifted Hellinger regularizers lead to a geometric rate 
(see left figure);
all other examples converge at a polynomial rate.}
\label{fig:examples}
\end{figure}


Albeit one-dimensional, the above examples provide a representative view of the geometry of Bregman proximal methods near a solution.
Specifically, they show that the divergence induced by a given regularizer may exhibit a very different behavior at the boundary of\;$\points$:
when $\sol$ is a boundary point, $\breg(\sol,\point)$ grows
as $\Theta(\norm{\point - \sol}^{2})$ in the Euclidean case,
as $\Theta(\norm{\point - \sol})$ for the negative entropy,
and, more generally,
as $\Theta(\norm{\point-\sol}^{\qexp})$ for the $\qexp$-th power regularizer.
As a result, when used as a measure of convergence, it is important to rescale the Bregman  divergence in order to avoid inflating \textendash\ or \emph{deflating} \textendash\ an algorithm's rate of convergence. 

Nonetheless, even if we take this rescaling into account, different instances of \eqref{eq:MD} may lead to completely different rates of convergence.
Specifically, in terms of absolute values (or norms), we observe a
geometric rate in the Euclidean and shifted Hellinger cases,
a rate of $\Theta(1/\run)$ for the negative entropy,
and
a rate of $\Theta(1/\run^{1/(2-\qexp)})$ for the $\qexp$-th power regularizer (\cf \cref{fig:examples} above).
This is due to the different first-order behavior of the iterative update map $\point\gets\fixmap(\point)$ that underlies \eqref{eq:MD}, which is itself intimately related to the growth rate of the Bregman divergence near a solution $\sol$ of \eqref{eq:VI}.
We make this relation precise in the next section.

\section{The Legendre exponent and convergence rate analysis}
\label{sec:general}

Our goal in this section is to provide a precise link between the geometry induced by a Bregman regularizer near a solution and the convergence rate of the associated Bregman proximal method.
The key notion in this regard is that of the \emph{Legendre exponent}, which we define and discuss in detail below.

\subsection{The Legendre exponent}
\label{sec:Legendre}

Our starting point is the observation that, without loss of generality, the local strong convexity requirement for $\hreg$ can be expressed as
\begin{equation}
\label{eq:Breg-lower}
\breg(\base,\point)
	\geq \tfrac{1}{2} \norm{\base - \point}^{2}
	\quad
	\text{\revise{for all $\point\in\proxdom$ sufficiently close to $\base$}}.
\end{equation}
Qualitatively, this means that the convergence topology induced by the Bregman divergence of $\hreg$ on $\points$ is \emph{at least as fine} as the ambient norm topology:
if a sequence $\curr[\point]\in\proxdom$, $\run=\running$, converges to $\base\in\points$ in the Bregman sense ($\breg(\base,\curr[\point]) \to 0$), it also converges in the ambient norm topology ($\norm{\curr[\point] - \base}\to0$).
On the other hand, from a quantitative standpoint, the rate of this convergence could be quite different:
as we saw in the previous section, the reverse inequality $\breg(\base,\point) = \bigoh(\norm{\base-\point}^{2})$ may fail to hold, in which case $\sqrt{\breg(\base,\curr[\point])}$ and $\norm{\point - \curr[\point]}$ would exhibit a different asymptotic behavior.

To quantify this gap, we use the notion of the \emph{Legendre exponent}, as introduced in \cite{AIMM21}.

\begin{definition}
\label{def:Legendre}
Let $\hreg$ be a Bregman regularizer on $\points$.
The \emph{Legendre exponent} of $\hreg$ at $\base\in\points$ is defined as
\begin{equation}
\label{eq:Legendre}
\legof{\base}
	\defeq \inf\setdef*{\legexp\in[0,1]}{\limsup_{\point\to\base} \frac{\sqrt{\breg(\base,\point)}}{\norm{\point-\base}^{1-\legexp}} < \infty}
\end{equation}
and we say that $\hreg$ is \emph{tight} at $\base$ if the infimum is attained in \eqref{eq:Legendre}, \ie if $\legof{\base}$ is the minimal $\legexp\in[0,1]$ such that
\begin{equation}
\label{eq:Breg-local}
\breg(\base,\point)
	= \bigof[\big]{\norm{\base - \point}^{2(1-\legexp)}}
	\quad
	\text{for $\point$ near $\base$}.
\end{equation}
\end{definition}

Informally, the Legendre exponent measures the deficit in relative size between ordinary ``norm neighborhoods'' in $\points$ and the corresponding ``Bregman neighborhoods'' induced by the sublevel sets of the Bregman divergence.
Specifically,
\begin{enumerate*}
[(\itshape i\hspace*{.5pt}\upshape)]
\item
the case $\legof{\base} = 0$ corresponds to the ``norm-like'' behavior $\breg(\base,\point) = \Theta(\norm{\base-\point}^{2})$;
\item
any other value $\legof{\base} \in (0,1)$ indicates a different limiting behavior for $\breg(\base,\point)$ as $\point\to\base$;
and, finally,
\item
when $\legof{\base} = 1$ we may have $\limsup_{\point\to\base} \breg(\base,\point) > 0$.
\end{enumerate*}
In this last case, the ambient norm topology is \emph{strictly coarser} than the Bregman topology in the sense that $\breg(\base,\curr)$ may remain bounded away from zero even if $\curr\to\base$ as $\run\to\infty$;
we provide an example of such behavior below \textendash\ and see also \cite{AIMM21,pauwels2023nature} for further discussion.

\begin{example}
[Non-compatible topologies]
Let $\points = \setdef{\point\in\R^{\nCoords}}{\twonorm{\point} \leq 1}$ be the unit Euclidean ball in $\R^{\nCoords}$ and consider the $\nCoords$-dimensional Hellinger regularizer $\hreg(\point) = -\sqrt{1 - \twonorm{\point}^{2}}$.
Then, for all $\base$ on the boundary of $\points$ and all $\point \in \proxdom = \intr\points$, we readily get
\begin{equation}
\label{eq:Breg-Hellinger}
\breg(\base,\point)
	= \frac{1 - \braket{\base}{\point}}{\sqrt{1 - \twonorm{\point}^{2}}}.
\end{equation}
If $\nCoords\geq2$, the limit $\lim_{\point\to\base} \breg(\base,\point)$ may not exist,
a fact which has the following counterintuitive consequences:
\begin{enumerate*}
[(\itshape i\hspace*{1pt}\upshape)]
\item
the ``Hellinger ball'' $\ball_{\radius}^{\hreg}(\base) \defeq \setdef{\point\in\proxdom}{\breg(\base,\point) \leq \radius^{2}/2}$ is \emph{not closed} in the Euclidean topology;
and
\item
the ``Hellinger center'' $\base$ of $\ball_{\radius}^{\hreg}(\base)$ actually belongs to the Euclidean boundary of $\ball_{\radius}^{\hreg}(\base)$.
\end{enumerate*}
As a result, for all $\nCoords\geq2$, it is straightforward to construct a sequence $\curr[\point]$ with $\twonorm{\curr[\point] - \base} \to 0$, but which remains at \emph{constant} Hellinger divergence relative to $\base$.
\footnote{For instance, if $\nCoords = 2$, the point $\point_{u} = (1-u,\sqrt{2u(1-u)})$ converges to $\base = (1,0)$ as $u\to0^{+}$, even though $\breg\left(\base,\point_{u}\right) = 1$ for all $u\in(0,1)$.
Crucially, if $\nCoords=1$, this phenomenon does not occur, \cf \cref{ex:Hell}.}
\hfill
\endenv
\end{example}

For illustration purposes, we compute below the Legendre exponent for each of the running examples of \cref{sec:examples} (see also \cref{tab:rates}):
\begin{enumerate}
\item
\emph{Quadratic regularization} (\cref{ex:Eucl}):
Since $\breg(\base,\point) = \parens{\base - \point}^{2}/2$ for all $\base,\point\in\points$, we have $\legof{\base} = 0$ for all $\base\in\points$.

\item
\emph{Negative entropy} (\cref{ex:ent}):
For $\base=0$, \cref{eq:mirror-ent} gives $\breg(0,\point) = \point$, so $\legof{0} = 1/2$.
Otherwise, for all $\base \in \proxdom = (0,\infty)$, a Taylor expansion with Lagrange remainder yields $\breg(\base,\point) = \bigoh(\parens{\base - \point}^{2})$, so $\legof{\base} = 0$ for all $\base\in(0,\infty)$.

\item
\emph{Tsallis entropy} (\cref{ex:frac}):
For $\base=0$, \cref{eq:mirror-frac} gives $\breg(0,\point) = \point^{\qexp}/\qexp$, so $\legof{0} = \max\{0,1-\qexp/2\}$.
Otherwise, for all $\base \in \proxdom = (0,\infty)$, a Taylor expansion yields $\breg(\base,\point) = \bigoh(\parens{\base - \point}^{2})$, so $\legof{\base} = 0$ in this case.

\item
\emph{Hellinger regularizer} (\cref{ex:Hell}):
For $\base = \pm1$, \cref{eq:mirror-Hell} readily gives $\breg(\pm1,\point) = \sqrt{(1 \mp \point)/(1 \pm \point)} = \Theta\parens{\abs{\point\mp1}^{1/2}}$, so $\legof{\pm1} = 1-1/4 = 3/4$.
Instead, if $\base\in(-1,1)$ a Taylor expansion again yields $\breg(\base,\point) = \bigoh(\parens{\base-\point}^{2})$, so $\legof{\base} = 0$ in this case.
\end{enumerate}
\smallskip

A common pattern that emerges above is that $\legof{\base}=0$ whenever $\base$ is an interior point.
We make this observation precise in \refinapp{lem:Leg-proxdom}{app:aux}, where we show more generally that $\legof{\base} = 0$ whenever $\nabla\hreg$ is (locally) Lipschitz continuous in a neighborhood of $\base$ in $\points$.%

\subsection{Convergence rate analysis}
\label{sec:rate-general}

We are now in a position to state our first general result for the convergence rate of \eqref{eq:BPM}.
To do so, we will make the blanket assumption that $\hreg$ is tight at $\sol$ with Legendre exponent $\legsol \defeq \legof{\sol}$.
In particular, this means that there exists a neighborhood $\legnhd$ of $\sol$ in $\points$ and a positive constant $\legconst>0$ such that
\begin{equation}
\label{eq:Breg-upper}
\breg(\sol,\point)
	\leq \frac{\legconst}{2} \norm{\point - \sol}^{2(1-\legsol)}
	\quad
	\text{for all $\point\in\legnhd$}.
\end{equation}
To ligthen notation, we will also assume that $\hreg$ is $1$-strongly convex on $\nhd$ (\cf \cref{rem:Bregman} and the beginning of \cref{sec:Legendre}).
We then have the following result.

\begin{theorem}
\label{thm:general}
Suppose that \cref{asm:Lipschitz,asm:strong,asm:signal-base} hold and \eqref{eq:BPM} is run with a constant step-size $\curr[\step] \equiv \step$, $\run = \running$, such that
\begin{equation}
\label{eq:step}
\step
	\leq \frac{1}{2\gold\lips}
	\quad
	\text{and}
	\quad
\step (1-\coef[a]-\coef[b])^{2}
	\leq \frac{\strong}{8\lips^{2}}
\end{equation}
where $\gold = (\sqrt{5}+1)/2$ is the golden ratio.
If $\init$ is initialized sufficiently close to $\sol$,
the iterates $\curr$ of \eqref{eq:BPM}
enjoys the bound
\begin{equation}
\label{eq:rate}
\breg(\sol,\curr)
	\leq \breg(\sol,\init) \cdot
	\begin{cases*}
		\parens*{1 - \frac{\strong\step}{2\legconst}}^{\run - 1}
			&\quad
			if $\legsol = 0$,
			\\[\medskipamount] 
		\bracks*{1 + \Const \strong \step(\run - 1)}^{1-1/\legsol}
			&\quad
			if $\legsol\in(0,1)$,
	\end{cases*}
\end{equation}
where
\(
\Const
= \expleg
\max\braces[\big]{2\legconst^{\frac{1}{1 - \legsol}}\breg(\sol,\init)^{-\expleg}, 2^{\expleg}}^{-1}.
\)
\end{theorem}

Before moving on to the proof of \cref{thm:general}, some remarks and corollaries are in order (see also \cref{tab:rates} for an explicit illustration of the derived rates for \crefrange{ex:Eucl}{ex:Hell}):
\smallskip


\begin{table}[tbp]
\footnotesize
\centering
\renewcommand{\arraystretch}{1.25}

\begin{tabular}{lcccc}
\toprule
	&\textbf{Domain ($\points$)}
	&\textbf{Regularizer ($\hreg$)}
	&\textbf{Legendre Exponent ($\legof{\base}$)}
	&\textbf{Convergence Rate}
	\\
\midrule
\scshape{Euclidean}	
	&arbitrary
	&$\point^{2}/2$
	&$0$
	&Linear 
	\\
\scshape{Entropic}
	&$[0,\infty)$
	&$\point\log\point$
	&$1/2$
	&$\bigoh(1/\run)$
	\\
\scshape{Tsallis}
	&$[0,\infty)$
	&$[\qexp(1-\qexp)]^{-1} (\point - \point^{\qexp})$
	&$\max\braces{0,1-\qexp/2}$
	&$\bigoh(1/\run^{\qexp/(2-\qexp)})$
	\\
\scshape{Hellinger}
	&$[-1,1]$
	&$-\sqrt{1-\point^{2}}$
	&$3/4$
	&$\bigoh(1/\run^{1/3})$
	\\
\bottomrule
\end{tabular}

\smallskip
\caption{Summary of the Legendre exponents for the $1$-dimensional examples of \Cref{sec:setup} at a boundary point $\base$ of $\points \subset \R$, and the associated convergence rates in terms of the Bregman divergence $\breg(\sol,\curr)$}.
\label{tab:rates}
\end{table}


\setcounter{remark}{0}
\begin{remark}
The first point of note is the sharp drop in the convergence rate of \eqref{eq:BPM} from geometric, when $\legsol=0$, to a power law when $\legsol>0$.
As we saw in \cref{sec:examples}, this drop is unavoidable, even when $\points$ is $1$-dimensional and $\vecfield$ is affine;
in fact, the calculations of \cref{sec:examples} show that the rates provided by \cref{thm:general} are, in general, unimprovable.
\hfill
\endenv
\end{remark}

\begin{remark}
We should also note that the guarantees of \cref{thm:general} are 
stated in terms of the Bregman divergence, not the ambient norm.
Since $\breg(\sol,\curr) = \Omega(\norm{\curr - \sol}^{2})$, these bounds can be restated in terms of $\norm{\curr-\sol}$, but this conversion is not without loss of information:
if the bound $\breg(\sol,\curr) = \Omega(\norm{\curr - \sol}^{2})$ is not tight, the actual rate in terms of the norm may be significantly different.
This phenomenon was already observed in the $1$-dimensional examples of \cref{sec:examples} where $\breg(\sol,\curr) = \Theta(\norm{\curr-\sol}^{2(1-\legsol)})$, in which case \cref{thm:general} gives
\begin{equation}
\norm{\curr-\sol}
	= \bigof[\big]{\run^{-1/(2\legsol)}}
\end{equation}
whenever $\legsol>0$ (see also \cref{tab:rates}).
In general however, the Bregman divergence may grow at different rates along different rays emanating from $\sol$, so it is not always possible to translate a Bregman-based bound to a norm-based bound (or vice versa).
This analysis requires a much closer look at the geometric structure of $\points$, depending on which constraints are active at $\sol$;
we examine this issue at depth in \cref{sec:sharp}.
\hfill
\endenv
\end{remark}

\begin{remark}
\label{rem:variable}
We should also note that, even though \cref{thm:general} is stated for a constant step-size, our proof allows for a variable step-size $\curr[\step]$, provided that the step-size conditions \eqref{eq:step} are satisfied.
In this case, the bounds \eqref{eq:rate} becomes
\begin{equation*}
\breg(\sol,\curr)
	\leq \breg(\sol,\init)
	\!\cdot\! 
	\begin{cases*}
	\prod_{\runalt=\start}^{\run-1}\parens*{1 - \frac{\strong\iter[\step]}{2\legconst}}
		&\text{if $\legsol=0$},
	\\
	\bracks[\big]{1 + \Const \strong \sum_{\runalt=\start}^{\run-1} \iter[\step]}^{1-1/\legsol}
		&\text{if $\legsol \in (0, 1)$)}.
		\;
	\end{cases*}
\end{equation*}
\end{remark}

\subsection{Proof of \cref{thm:general}}
\label{sec:proof-general}

We now proceed to the proof of \cref{thm:general}, beginning with a series of intermediate results tailored to the update structure of \eqref{eq:BPM}.
The first of these lemmas relates the Bregman divergence before and after a prox-step modulo an element of the polar cone $\pcone(\base) \defeq \setdef{\dbase\in\dpoints}{\braket{\dbase}{\point - \base} \leq 0 \; \text{for all $\point\in\points$}}$ of $\points$ at the reference point $\base$.

\begin{lemma}
\label{lem:onestep}
Let $\new = \proxof{\point}{\dvec}$ for some $\point\in \nhd \cap \proxdom$, $\dvec\in\dpoints$, such that $\new\in\nhd$.
Then, for all $\base\in\points$ and all 
$\dbase\in\pcone(\base)$, we have:
\begin{subequations}
\begin{align}
\breg(\base,\new)
	&\leq \breg(\base,\point)
		+ \braket{\dvec - \dbase}{\new - \base}
		- \breg(\new,\point)
		\\
	&\leq \breg(\base,\point)
		+ \braket{\dvec - \dbase}{\point - \base}
		+ \tfrac{1}{2} \dnorm{\dvec - \dbase}^{2}
		\,,
\end{align}
\end{subequations}
\end{lemma}

The next lemma extends \cref{lem:onestep} to emulate the two-step structure of \eqref{eq:BPM}:

\begin{lemma}
\label{lem:twostep}
\revise{Let $\new_{i} = \proxof{\point}{\dvec_{i}}$ for some $\point\in \nhd \cap \proxdom$, $\dvec_{i}\in\dpoints$, such that $\new_i \in \nhd$, $i=1,2$.}
Then, for all $\base\in\points$ and all $\dbase\in\pcone(\base)$, we have:
\begin{equation}
\breg(\base,\new_{2})
	\leq \breg(\base,\point)
		+ \braket{\dvec_{2} - \dbase}{\new_{1} - \base}
		+ \tfrac{1}{2} \dnorm{\dvec_{2} - \dvec_{1} - \dbase}^{2}
		- \tfrac{1}{2} \norm{\new_{1} - \point}^{2}.
\end{equation}
\end{lemma}

Versions of the above inequalities already exist in the literature, see \eg\cite[Lem.~4]{JNT11}, \cite[Prop.~B.4]{MLZF+19}.
The novelty in \cref{lem:onestep,lem:twostep} is the extra term involving the polar vector $\dbase\in\pcone(\base)$;
this term plays an important role in the sequel, so we provide a complete proof \refapp{app:aux}.

With these preliminaries in hand, we proceed to derive two further inequalities that play a pivotal role in the analysis of \eqref{eq:BPM}.
The first is an immediate corollary of \cref{lem:twostep}:

\begin{corollary}
\label{cor:template}
Let $\sol$ be a solution of \eqref{eq:VI}
Then, for all $\coef[c]\geq0$ and all $\run=\running$ \revise{such that $\curr, \lead \in \nhd$} the iterates of \eqref{eq:BPM} satisfy the template inequality
\begin{align}
\label{eq:template}
\breg(\sol,\next)
	\leq \breg(\sol, \curr)
		&- \curr[\step] \braket{\lead[\signal] - \coef[c]\solvec}{\lead - \sol}
	\notag\\
		&+ \tfrac{1}{2} \curr[\step]^{2} \dnorm{\lead[\signal] - \curr[\signal] - \coef[c]\solvec}^{2}
		- \tfrac{1}{2} \norm{\lead - \curr}^{2}.
\end{align}
\end{corollary}

\begin{proof}
Since $\sol$ is a solution of \eqref{eq:VI}, we have $\solvec \in -\pcone(\sol)$.
\cref{eq:template} then follows by invoking \cref{lem:twostep} with $\point \gets \curr$, $\base\gets\sol$, $\dbase \gets - \coef[c]\curr[\step]\solvec \in \pcone(\sol)$ and $(\dvec_{1},\dvec_{2}) \gets (-\curr[\step]\curr[\signal],-\curr[\step]\lead[\signal])$.
\end{proof}

The second inequality that we derive provides an ``energy function'' for \eqref{eq:BPM}, namely
\begin{equation}
\label{eq:energy}
\curr[\energy]
	= \curr[\breg] + \curr[\pot]
\end{equation}
where $\curr[\breg]
	= \breg(\sol,\curr)$
	and
$\curr[\pot] = \prev[\step]^{2} \dnorm{(\coef[a]+\coef[b]) \beforelead[\signal] - \prev[\signal]}^{2}$ (by convention, we take $\init[\pot] = 0$).
The lemma below outlines the Lyapunov properties of $\curr[\energy]$.

\begin{proposition}
\label{prop:energy}
Suppose that \cref{asm:Lipschitz,asm:signal-base} hold and \eqref{eq:BPM} is run with a step-size such that
\begin{equation}
\label{eq:step-energy}
\coef \curr[\step] + 4\curr[\step]^{2}\lips^{2}
	\leq 1
	\quad
	\text{for some $\coef\geq0$ and all $\run=\running$}
\end{equation}
Then the iterates $\curr$ of \eqref{eq:BPM} satisfy \revise{for $\run \geq \start$ such that $\curr,\lead \in \nhd$},
\begin{align}
\label{eq:energy-bound}
\next[\energy]
	\leq \curr[\energy]
		- \coef\curr[\step] \curr[\pot]
	&- \curr[\step] \braket{\vecfield(\lead) - \solvec}{\lead - \sol}
	\notag\\
	&- \curr[\step] (\coef[a] + \coef[b]) \braket{\solvec}{\lead - \sol}
		- \tfrac{1}{2} \norm{\lead - \curr}^{2}
	\notag\\
	&+ \curr[\step]^{2} (1-\coef[a]-\coef[b])^{2} \lips^{2} \norm{\lead- \sol}^{2}
	\notag\\
	&+ 2\curr[\step]^{2}(\coef[a] + \coef[b])^{2} \lips^{2} \norm{\lead - \curr}^{2}.
\end{align}
\end{proposition}

\begin{proof}
Let $\coef[c] = 1 - \coef[a] - \coef[b]$ so $\coef[c] \geq 0$ by \cref{asm:signal-base}.
\cref{cor:template} then yields
\begin{align}
\label{eq:energy-1}
\next[\breg]
	\leq \curr[\breg]
	&- \curr[\step] \braket{\lead[\signal] - \solvec}{\lead - \sol}
	\notag\\
	&- \curr[\step] (\coef[a] + \coef[b]) \braket{\solvec}{\lead - \sol}
		- \tfrac{1}{2} \norm{\lead - \curr}^{2}
	\notag\\
	&+ \frac{\curr[\step]^{2}}{2} \dnorm{\lead[\signal] - \curr[\signal] - \coef[c]\solvec}^{2}.
\end{align}
Since $\lead[\signal] = (\coef[a] + \coef[b])\lead[\signal] + \coef[c]\lead[\signal]$, the last term above may be bounded as
\begin{align}
\label{eq:energy-2}
\tfrac{1}{2} \curr[\step]^{2} \dnorm{\lead[\signal] - \curr[\signal] - \coef[c]\solvec}^{2}
	&\leq \curr[\step]^{2}\coef[c]^{2} \dnorm{\lead[\signal] - \solvec}^{2}
		+ \curr[\step]^{2}\dnorm{(\coef[a]+\coef[b])\lead[\signal] - \curr[\signal]}^{2}
	\notag\\
	&\leq \curr[\step]^{2}\coef[c]^{2}\lips^{2} \dnorm{\lead- \sol}^{2}
		+ \curr[\step]^{2}\dnorm{(\coef[a]+\coef[b])\lead[\signal] - \curr[\signal]}^{2}
	\notag\\
	&= \curr[\step]^{2} (1-\coef[a]-\coef[b])^{2}\lips^{2} \norm{\lead- \sol}^{2} + \next[\pot]\,,
\end{align}
where we used \cref{asm:Lipschitz} in the second line and the definition \eqref{eq:energy} of $\curr[\pot]$ in the last one.
Thus, combining \cref{eq:energy-1,eq:energy-2} and comparing to \eqref{eq:energy-bound}, it suffices to show that
\begin{equation}
\label{eq:potbound}
2\next[\pot]
	\leq (1-\coef\curr[\step]) \curr[\pot]
		+ 4\curr[\step]^{2}(\coef[a] + \coef[b])^{2} \lips^{2} \norm{\lead - \curr}^{2}
	\quad
	\text{for all $\run=\running$}
\end{equation}
We consider two distinct cases for this below.

\para{Case 1: $\run=\start$}
By the definition \eqref{eq:energy} of $\curr[\pot]$ and \cref{eq:signal-base,eq:signal-lead}, we have:
\begin{align*}
\afterinit[\pot]
	= \init[\step]^{2} \dnorm{(\coef[a]+\coef[b])\signal_{3/2} - \init[\signal]}^{2}
	&= \init[\step]^{2} (\coef[a]+\coef[b])^{2} \dnorm{\vecfield(\state_{3/2}) - \vecfield(\init)}^{2}
	\notag\\
	&\leq \init[\step]^{2} (\coef[a]+\coef[b])^{2} \lips^{2} \norm{\state_{3/2} - \init}^{2},
\end{align*}
where we used the initialization assumption $\init = \state_{1/2}$ in the second equality and the Lipschitz continuity of $\vecfield$ in the last one.
Since $\init[\pot]=0$ by construction, our claim is immediate.

\para{Case 2: $\run>\start$}
By Young's inequality and the Lipschitz continuity of $\vecfield$, we readily obtain
\begin{align*}
\next[\pot]
	&= \curr[\step]^{2} \dnorm{(\coef[a]+\coef[b])\lead[\signal] - \curr[\signal]}^{2}
	\notag\\
	&= \curr[\step]^{2} \dnorm[\big]{
			(\coef[a]+\coef[b]) \bracks{\vecfield(\lead) - \vecfield(\curr)}
			+ \coef[b] \bracks{\vecfield(\curr) - \vecfield(\beforelead)}}^{2}
	\notag\\
	&\leq 2\curr[\step]^{2} (\coef[a] + \coef[b])^{2} \dnorm{\vecfield(\lead) - \vecfield(\curr)}^{2}
		+ 2\curr[\step]^{2} \coef[b]^{2} \dnorm{\vecfield(\curr) - \vecfield(\beforelead)}^{2}
	\notag\\
	&\leq 2\curr[\step]^{2} (\coef[a] + \coef[b])^{2} \lips^{2} \norm{\lead - \curr}^{2}
		+ 2\curr[\step]^{2} \coef[b]^{2} \lips^{2} \norm{\curr - \beforelead}^{2}
	\notag\\
	&\leq 2\curr[\step]^{2} (\coef[a] + \coef[b])^{2} \lips^{2} \norm{\lead - \curr}^{2}
		+ 2\curr[\step]^{2} \coef[b]^{2} \lips^{2} \prev[\step]^{2} \dnorm{\beforelead[\signal] - \prev[\signal]}^{2}
\end{align*}
where, in the last line, we used \cref{lem:proxlip} to bound the difference $\curr - \beforelead$ as
\begin{equation}
\norm{\curr - \beforelead}
	= \norm{\proxof{\prev}{-\prev[\step]\beforelead[\signal]} - \proxof{\prev}{-\prev[\step]\prev[\signal]}}
	\leq \prev[\step] \dnorm{\beforelead[\signal] - \prev[\signal]}.
\end{equation}
Finally, by \cref{asm:signal-base}, we have $\coef[c] = 0$ whenever $\coef[b] > 0$, so
$\coef[b]^{2} \prev[\step]^{2}\dnorm{\beforelead[\signal] - \prev[\signal]}^{2}
	= \coef[b]^{2} \prev[\step]^{2} \dnorm{(1 - \coef[c])\beforelead[\signal] - \prev[\signal]}^{2} = \coef[b]^{2}\curr[\pot]$ for all $\run>\start$.
Hence, putting everything together, we get
\begin{equation}
\next[\pot]
	\leq 2\curr[\step]^{2} (\coef[a] + \coef[b])^{2} \lips^{2} \norm{\lead - \curr}^{2}
		+ 2\curr[\step]^{2} \coef[b]^{2} \lips^{2} \curr[\pot].
\end{equation}
\cref{eq:potbound} then follows by the requirement \eqref{eq:step-energy}, which implies that $2\curr[\step]^{2} \lips^{2} \leq (1 - \coef\curr[\step])/2$.
\end{proof}

Moving forward, since $\sol$ is a solution of \eqref{eq:VI},
the first line of \eqref{eq:energy-bound} yields a negative $\bigoh(\curr[\step])$ contribution to $\curr[\energy]$, 
whereas the third and fourth lines collectively represent a subleading $\bigoh(\curr[\step]^{2})$ ``error term''. 
This decomposition would suffice for the analysis of \eqref{eq:BPM} if the coupling term $\braket{\solvec}{\lead - \sol}$ did not incur an additional $\bigoh(\curr[\step])$ positive contribution to $\next[\energy]$.
This error term is difficult to control but if $\sol$ satisfies \eqref{eq:strong}, we have the following bound.

\begin{lemma}
\label{lem:strong-bound}
Suppose that \cref{asm:strong} holds.
Then, for all $\point\in\points$, $\pointalt\in\basin$ and all $\coef[c]\in[0,1]$, we have:
\begin{equation}
\label{eq:strong-bound}
\braket{\vecfield(\pointalt) - \coef[c] \solvec}{\pointalt - \sol}
	\geq \tfrac{1}{2} \strong \norm{\point - \sol}^{2}
		- \strong \norm{\pointalt - \point}^{2}.
	\end{equation}
\end{lemma}

\begin{proof}
Since $\sol$ is a solution of \eqref{eq:VI} and $\coef[c]\in[0,1]$, we have $(1-\coef[c]) \braket{\solvec}{\pointalt - \sol} \geq 0$ for all $\pointalt\in\points$.
Hence, by \cref{asm:strong}, we get
\begin{equation}
\braket{\vecfield(\pointalt) - \coef[c]\solvec}{\pointalt - \sol}
	\geq \braket{\vecfield(\pointalt) - \solvec}{\pointalt - \sol}
	\geq \strong \norm{\pointalt -\sol}^{2}
\end{equation}
and our assertion follows from the basic bound $\norm{\point - \sol}^{2} \leq 2\norm{\point - \pointalt}^{2} + 2\norm{\pointalt - \sol}^{2}$.
\end{proof}

With this ancillary estimate in hand, we may finally sharpen \cref{prop:energy} to obtain a bona fide energy inequality for solutions satisfying \eqref{eq:strong}:

\begin{proposition}
\label{prop:energy-strong}
Suppose that \cref{asm:Lipschitz,asm:strong,asm:signal-base} hold and \eqref{eq:BPM} is run with\;$\curr[\step]$\;such\;that%
\begin{equation}
\label{eq:step-energy-strong}
2\strong \curr[\step] + 4\curr[\step]^{2}\lips^{2}
	\leq 1
	\quad
	\text{and}
	\quad
(1-\coef[a]-\coef[b])^{2}\curr[\step]
	\leq \frac{\strong}{8\lips^{2}}
	\quad
	\text{for all $\run=\running$}
\end{equation}
Then, for all $\run \geq \start$ such that \revise{$\curr,\lead \in \nhd$} and $\lead \in \basin$, we have
\begin{equation}
\label{eq:energy-strong}
\next[\energy]
	\leq \curr[\energy]
		- \strong \curr[\step] \curr[\pot]
		- \tfrac{1}{4} \strong\curr[\step] \norm{\curr - \sol}^{2}.
\end{equation}
\end{proposition}

\begin{proof}
Assume that $\lead\in\basin$ and set $\coef[c] = 1 - \coef[a] - \coef[b]$.
Then, invoking \cref{lem:strong-bound} with $\point \gets \curr$ and $\pointalt \gets \lead$, we get
\begin{flalign}
\MoveEqLeft
\braket{\vecfield(\lead) - \solvec}{\lead - \sol}
	+ (\coef[a] + \coef[b]) \braket{\solvec}{\lead-\sol}
	\notag\\
	&\geq \tfrac{1}{2} \strong \norm{\curr - \sol}^{2}
		- \strong \norm{\lead - \curr}^{2}.
\end{flalign}
Thus, taking $\coef \gets \strong$ in \cref{prop:energy} (in terms of step-size conditions, the first part of \eqref{eq:step-energy-strong} implies \eqref{eq:step-energy}) and combining with the above, the bound \eqref{eq:energy-bound} becomes
\begin{align}
\label{eq:strong-bound-energy-strong-proof}
\next[\energy]
	\leq \curr[\energy]
		- \strong \curr[\step] \curr[\pot]
		&- \tfrac{1}{2} \strong\curr[\step] \norm{\curr - \sol}^{2}
		+ \curr[\step]^{2} \coef[c]^{2} \lips^{2} \norm{\lead- \sol}^{2}
	\notag\\
	&- \tfrac{1}{2} \parens[\big]{1 - 4 \curr[\step]^{2} (\coef[a] + \coef[b])^{2} \lips^{2} - 2\strong \curr[\step]} \norm{\lead - \curr}^{2}.
\end{align}
Hence, writing $\norm{\lead - \sol}^{2} \leq 2\norm{\lead-\curr}^{2} + 2\norm{\curr-\sol}^{2}$ and rearranging, we obtain
\begin{align}
\label{eq:energy-strong-proof}
\next[\energy]
	\leq \curr[\energy]
		- \strong \curr[\step] \curr[\pot]
		&- \tfrac{1}{2} \parens[\big]{\strong \curr[\step] - 4 \curr[\step]^{2}\coef[c]^{2}\lips^{2}}
		\, \norm{\curr - \sol}^{2}
	\notag\\
	&- \tfrac{1}{2} \parens[\big]{1 - 4 \curr[\step]^{2} \parens{(\coef[a] + \coef[b])^{2} + \coef[c]^{2}} \lips^{2} - 2\strong \curr[\step]}
		\, \norm{\lead - \curr}^{2}.
\end{align}
Since $\coef[a],\coef[b],\coef[c] \geq 0$ and $\coef[a]+\coef[b]+\coef[c]=1$, we also have $(\coef[a]+\coef[b])^{2} + \coef[c]^{2} \leq 1$, so the step-size assumption \eqref{eq:step-energy-strong} guarantees that the last term in \eqref{eq:energy-strong-proof} is nonpositive.
Likewise, the second part of \eqref{eq:step-energy-strong} gives $\strong\curr[\step] - 4\curr[\step]^{2}\coef[c]^{2}\lips^{2} \geq \tfrac{1}{2}\strong\curr[\step]$, so the energy inequality \eqref{eq:energy-strong} follows and our proof is complete.
\end{proof}

We finally have all the required building blocks in place to prove \cref{thm:general}.

\begin{proof}[Proof of \cref{thm:general}]
Our proof strategy consists of the following basic steps:
\begin{enumerate}
\item
	We first show that, if the step-size of \eqref{eq:BPM} satisfies \eqref{eq:step} and $\init$ is initialized sufficiently close to $\sol$, \revise{the base and leading state sequences $\curr$ and $\lead$, $\run=\running$, both remain within the neighborhood $\legnhd\cap\basin$} of $\sol$ where \eqref{eq:strong}, \eqref{eq:Breg-lower}, and \eqref{eq:Breg-upper} all hold.

\item
By virtue of this stability result, the energy inequality \eqref{eq:energy-strong} and the definition of the Legendre exponent allow us to express $\curr[\breg] = \breg(\sol,\curr)$ as $\next[\breg] \leq \curr[\breg] - \bigof[\big]{\curr[\breg]^{1/(1-\legsol)}}$ up to an error term that vanishes at a geometric rate.
The rates \eqref{eq:rate} are then derived by analyzing this recursive inequality for $\legsol = 0$ and $\legsol>0$ respectively.
\end{enumerate}
We now proceed to detail the two steps outlined above.

\para{Step 1: Stability}
Take $\radius > 0$ such that $\ball_{\radius}^\points(\sol) \defeq \{ \point\in\points : \norm{\point - \sol}
\leq \radius \} \subset \revise{\basin \cap \nhd}$
\revise{and such that $\proxof{\point}{- \step \vecfield(\pointalt)}$ belongs to $\nhd$ for all $\point\in\ball_{\radius}^\points(\sol) \cap \proxdom$, $\pointalt \in \ball_{\radius}^\points(\sol)$, and all admissible step-sizes $\step$.
This is indeed possible by the continuity of the prox-mapping (see \refinapp{lem:proxlip}{app:aux}) and of $\vecfield$.}
Assume further that $\state_{1/2} = \init \revise{\in \ball_\radius^\points(\sol)}$ is such that $\breg(\sol,\state_{1/2}) = \breg(\sol,\init) \leq (1 - \coef) \radius^{2}/4$, where $\coef\in(0,1)$ is a constant to be determined later.
that
\begin{equation}
\label{eq:stability}
\revise{\max\braces{\norm{\beforelead - \sol}, \norm{\curr - \sol}}}
	\leq \radius
	\quad
	\text{and}
	\quad
\curr[\energy]
	\leq \prev[\energy],
\end{equation}
which will show in particular that $\lead\in\basin\cap\nhd$ for all $\run\geq\start$.
Indeed:

\begin{itemize}
\addtolength{\itemsep}{\smallskipamount}
\item
	\revise{For the base case ($\run=\start$), we have $\state_{1/2} = \state_1 \in \ball_{\radius}^\points(\sol)$}
and $\init[\energy] = \beforeinit[\energy]$ by construction, so there is nothing to show.
\item
For the induction step, assume \eqref{eq:stability} holds.
Then, \revise{since $\curr \in \ball_{\radius}^\points(\sol) \subset \nhd$,}
\eqref{eq:Breg-lower} yields
\begin{equation}
\label{eq:curr-bound}
\tfrac{1}{2} \norm{\curr - \sol}^{2}
	\leq \curr[\breg]
	\leq \curr[\energy]
	\leq \init[\energy]
	= \breg(\sol,\init)
\end{equation}
\revise{
Moreover,  both $\curr$ and $\beforelead$ are in $\ball_{\radius}^\points(\sol)$ so that, by construction, $\lead$ is still in $\nhd$.}
Now, to show that $\lead \in \ball_{\radius}^\points(\sol)$, \cref{lem:onestep} with
$\base\gets\sol$, $\point \gets \curr$, $\dvec \gets -\curr[\step]\curr[\signal]$ and $\dbase\gets - (\coef[a] + \coef[b])\curr[\step]\solvec$ gives
\begin{alignat}{2}
\lead[\breg]
	&\leq \curr[\breg]
		&&- \curr[\step] \braket{\curr[\signal] - (\coef[a]+\coef[b])\solvec}{\lead - \sol}
	\notag\\
	&\leq \curr[\breg]
		&&- \coef[a] \curr[\step] \braket{\vecfield(\curr) - \solvec}{\lead - \sol}
	\notag\\
	&
		&&- \coef[b] \curr[\step] \braket{\vecfield(\beforelead) - \solvec}{\lead - \sol}
\end{alignat}
and hence, by Young's inequality and \eqref{eq:Breg-lower}, we get
\begin{align}
\tfrac{1}{2} \norm{\lead - \sol}^{2}
	\leq \curr[\breg]
		&+ \curr[\step]^{2} \coef[a] \dnorm{\vecfield(\curr) - \solvec}^{2}
		+ \curr[\step]^{2}\coef[b] \dnorm{\vecfield(\beforelead) - \solvec}^{2}
	\notag\\
		&+ \tfrac{1}{4} (\coef[a]+\coef[b]) \norm{\lead - \sol}^{2}\,.
\end{align}
Since $ \coef[a]+\coef[b] \leq 1$, using \cref{asm:Lipschitz} and rearranging gives
\begin{align}
\label{eq:checkme}
\norm{\lead - \sol}^{2}
	&\leq 4 \curr[\breg]
		+ 4\curr[\step]^{2} \lips^{2}
			\max\braces{\norm{\curr - \sol}^{2}, \norm{\beforelead-\sol}^{2}}
	\notag\\
	&\leq (1-\coef) \radius^{2}
		+ 4\curr[\step]^{2} \lips^{2} \radius^{2}
\end{align}
where we used the fact that $\norm{\beforelead - \sol}^{2} \leq \radius^{2}$
and
$\norm{\curr - \sol}^{2} \leq 2\curr[\breg] \leq \frac{1}{2} (1-\coef)\radius^{2}$
(by the inductive hypothesis and \eqref{eq:curr-bound} respectively).
Thus, with $2\curr[\step]\lips \leq 1/\gold < 1$ by assumption, choosing $\coef = 1/\gold^{2}$ gives $\norm{\lead - \sol}^{2} \leq \radius^{2}$, which completes the first part of the induction.
Finally, for the second part,
\revise{we have $\next\in\nhd$ because $\curr,\lead$ have been shown to be in $\ball_{\radius}^\points(\sol)$}
and our step-size assumption gives
\begin{equation}
2\strong \curr[\step] + 4 \curr[\step]^{2}\lips^{2}
	\leq 2 \curr[\step]\lips + 4 \curr[\step]^{2}\lips^{2}
	\leq 1/\gold + 1/\gold^{2}
	= 1.
\end{equation}
Thus, since $\lead\in\basin$, \cref{prop:energy-strong} readily gives
\begin{equation}
\label{eq:descent}
\next[\energy]
	\leq \curr[\energy]
		- \strong \curr[\step] \curr[\pot]
		- \tfrac{1}{4} \strong \curr[\step] \norm{\curr - \sol}^{2}
	\leq \curr[\energy],
\end{equation}
and the induction is complete.
\end{itemize}

\para{Step 2: Convergence rate analysis}
From \eqref{eq:descent} and the local Legendre bound \eqref{eq:Breg-upper}, we get
\begin{equation}
\label{eq:descent-Leg}
\next[\energy]
	\leq \curr[\energy]
		- \strong \curr[\step] \curr[\pot]
		- \frac{\strong \curr[\step]}{2^{1-\leg} \legconst^{1+\leg}} \curr[\breg]^{1+\leg}
  \qquad \text{with $\leg = \legsol/(1 - \legsol)$.}
\end{equation}
We now distinguish two cases, depending on whether $\legsol=0$ or $\legsol>0$.

\begin{enumerate}
[left=1em,label={\bfseries Case \arabic*:}]
\item
If $\legsol = 0$, we have $\leg=0$ by definition and $\legconst\geq1$ by \eqref{eq:Breg-lower}.
\cref{eq:descent-Leg} then gives
\begin{equation}
\next[\energy]
	\leq \curr[\energy]
		- \frac{\strong\curr[\step]}{2\legconst} \curr[\breg]
		- \strong \curr[\step] \curr[\pot]
	\leq \parens*{1 - \frac{\strong\curr[\step]}{2\legconst}} \curr[\energy]
\end{equation}
so the case $\legexp=0$ of \eqref{eq:rate} follows immediately by setting $\curr[\step] \equiv \step$ for all $\run$.

\item
If $\legsol > 0$, then $\leg > 0$ too, so we will proceed by rewriting all terms in \cref{eq:descent-Leg} in terms of $\curr[\energy]$.
To that end, we have:
\begin{align}
\next[\energy]
	&\leq \curr[\energy]
		- \strong \curr[\step] \curr[\pot]
		- \frac{\strong \curr[\step]}{2^{1-\leg} \legconst^{1+\leg}} \curr[\breg]^{1+\leg}
	\notag\\
	&\leq \curr[\energy]
		- \frac{\strong \curr[\step]}{\breg(\sol,\init)^{\leg}} \curr[\pot]^{1+\leg}
		- \frac{\strong \curr[\step]}{2^{1-\leg} \legconst^{1+\leg}} \curr[\breg]^{1+\leg}
	\notag\\
	&\leq \curr[\energy]
		- \frac{\strong\curr[\step]}{\max(2^{1-\leg}\legconst^{1+\leg}, \breg(\sol,\init)^\leg)}
			\bracks*{\breg(\sol, \curr)^{1+\leg}+\curr[\pot]^{1+\leg}}
	\notag\\
	&\leq \curr[\energy]
		- \frac{\strong\curr[\step]}{\max(2\legconst^{1+\leg},2^{\leg}\breg(\sol,\init)^\leg)}
			\curr[\energy]^{1+\leg}
\end{align}
where,
in the second line, we used \eqref{eq:stability} to get $\curr[\pot] \leq \breg(\sol,\curr) + \curr[\pot] \leq \breg(\sol,\init)$,
and,
in the last line, we used the convexity of $\point^{1+\leg}$.
The case $\legexp\in(0,1)$ of \eqref{eq:rate} then follows from Lemma 6 of \cite[p.~46]{Pol87} (recreated \revise{as }\refinapp{lem:Polyak}{app:aux}).
\end{enumerate}
\end{proof}

\section{Finer results for linearly constrained problems}
\label{sec:sharp}

Our goal in this last section is to take a closer look at the convergence rate of \eqref{eq:BPM} for different solution configurations that arise in linearly constrained problems.
To that end, we begin by revisiting the examples of \cref{sec:examples}.

\subsection{Motivating examples, redux}
\label{sec:examples-sharp}

A common feature of \crefrange{ex:Eucl}{ex:Hell} is that the problem's defining vector field vanishes at the solution point under study.
In the series of examples below, we examine the rate of convergence achieved when this is not the case.


\begin{example}[Euclidean regularization]
\label{ex:Eucl-sharp}
Consider again the quadratic regularizer of \cref{ex:Eucl} over $\points = [0,\infty)$, but with $\vecfield(\point) = \point + 1$.
The solution of \eqref{eq:VI} is still $\sol = 0$ but the update \eqref{eq:MD-generic} now becomes
\begin{equation}
\label{eq:MD-Eucl-sharp}
\fixmap(\point)
	= \pospart{\point - \step(\point + 1)}
	= \pospart{(1-\step) \point -\step}\,.
\end{equation}
Since $\fixmap(\point) = 0$ for all sufficiently small $\point>0$, we readily conclude that $\curr$ converges to $\sol$ in a \emph{finite} number of iterations.
\hfill
\endenv
\end{example}



\begin{example}[Entropic regularization]
\label{ex:ent-sharp}
Under the entropic regularizer of \cref{ex:ent}, and taking again $\vecfield(\point) = \point+1$, the update rule \eqref{eq:MD-generic} becomes
\begin{equation}
\label{eq:MD-ent-sharp}
\fixmap(\point)
	= \point \exp(-\step(\point + 1))
	= \point e^{- \step} + o(\point)
	\sim \point e^{-\step}
\end{equation}
\ie $\fixmap$ is a contraction for small $\point > 0$.
Hence, in contrast to \cref{ex:ent}, $\curr$ converges to $0$ at a geometric rate, even though the problem's solution lies on the boundary of $\points$.
\hfill
\endenv
\end{example}



\begin{example}[Fractional power]
\label{ex:frac-sharp}
Finally, consider the fractional power regularizer of \cref{ex:frac}, again with $\vecfield(\point) = \point + 1$.
Then, for $\qexp \in (0,1)$, the update rule \eqref{eq:MD-generic} gives
\begin{align}
\fixmap(\point)
	&= \bracks{\point^{\qexp -1} + \step(1-\qexp)(\point + 1)}^{1/(\qexp-1)}
	= \point - \step\point^{2-\qexp} + o(\point^{2-\qexp})
\end{align}
for small $\point>0$.
Thus, by \cref{lem:basicnum}, we get 
that $\curr$ converges to $0$ as $\abs{\curr - \sol} = \Theta\parens[\big]{\run^{-1/(1-\qexp)}}$ and 
$\breg(\sol,\curr) = \Theta\parens[\big]{\run^{-\qexp/(1-\qexp)}}$, which is again faster than the rate 
given by \cref{thm:general}.
\hfill
\endenv
\end{example}


\Crefrange{ex:Eucl-sharp}{ex:frac-sharp} show that the convergence rate of \eqref{eq:BPM} \revise{when $\sol$ in the boundary} can change drastically depending on whether $\vecfield(\sol)$ is zero or not.
In the example below, we examine in more detail the behavior of the individual coordinates of $\curr$ as a function of the position of $\vecfield(\sol)$ relative to $\points$.


\begin{example}[Higher-dimensional simplices]
\label{ex:simplex-2d}
Consider the canonical two-di\-men\-sional simplex $\points = \setdef{(\point_{1}, \point_{2}, \point_3)\in \R_{+}^{3}}{\point_{1} + \point_{2} + \point_{3} = 1}$ of $\R^{3}$
equipped with the entropic regularizer
$\hreg(\point) = \sum_{\coord=1}^{3} \point_{\coord} \log\point_{\coord}$.
Consider also the vector field $\vecfield(\point) = \point - \base$ with $\base=(-\slack_{1}, -\slack_{2},1)$ for some $\slack_{1}, \slack_{2} \geq 0$, so the solution of \eqref{eq:VI} is $\sol = (0,0,1)$, an extreme point of $\points$.

Since the Legendre exponent of $\hreg$ at $\sol$ is easily seen to be $\legof{\sol} = 1/2$, \cref{thm:general} would indicate a rate of convergence of $\breg(\sol,\curr) = \bigoh(1/\run)$ or, in terms of norms, $\norm{\curr - \sol} = \bigoh(1/\run)$.
However, this rate can be very pessimistic if, for example, $\slack_{1} > 0$.
Indeed, in this case, since $\curr$ converges to $\sol = (0,0,1)$, the relevant coordinates of $\vecfield(\curr)$ will evolve as $\vecfield_{1}(\curr) = \state_{1,\run} + \slack_{1} = \slack_{1} + o(1)$ and $\vecfield_{3}(\curr) = \state_{3,\run} - 1 = o(1)$.
Accordingly, since entropic regularization on the simplex leads to the exponential weights update \cite{BecTeb03}
\begin{equation}
\label{eq:EW}
\state_{\coord,\run+1}
	\propto \state_{\coord,\run} \exp\parens*{-\step \vecfield_{\coord}(\curr)}
	\quad
	\text{for all $\run \geq \start$, $\coord = 1,2,3$},
\end{equation}
the fact that $\lim_{\run\to\infty}\state_{3,\run} = 1$ readily yields
\begin{align}
\state_{1,\run+1}
	\sim \frac{\state_{1,\run+1}}{\state_{3,\run+1}}
	&= \frac{\state_{1,\run}}{\state_{3,\run}}
		\exp\parens*{-\step \vecfield_{1}(\curr) + \step \vecfield_{3}(\curr)}
	= \frac{\state_{1,\run}}{\state_{3,\run}}
		\exp\parens*{-\step \slack_{1} + o(1)}
\end{align}
\ie $\state_{1,\run}$ converges to $0$ at a \emph{geometric} rate whenever $\slack_{1} > 0$.

By symmetry, the argument above yields the same rate for $\state_{2,\run}$ if $\slack_{2} > 0$.
However, as we show \refapp{app:ex}, if $\slack_{2} = 0$, we would have $\state_{2,\run} = \Theta(1/\run)$ no matter the value of $\slack_{1}$ (and likewise for the rate of $\state_{1,\run}$ if $\slack_{1}=0$).
In other words, the rate provided by \cref{thm:general} is tight for the coordinate $\coord \in \{1,2\}$ with a vanishing drift coefficient $\slack_{\coord}$, but not otherwise;
we will devote the rest of this section to deriving a formal statement (and proof) of the general principle underlying this observation.
\hfill
\endenv
\end{example}


\subsection{Linearly constrained problems}

For concreteness, we focus below on linearly constrained problems 
\revise{ where the different convergence behaviors outlined in the previous examples can be characterized in a precise manner.}
To lighten notation, we identify $\pspace$ with $\R^{\nCoords}$ endowed with the Euclidean scalar product $\braket{\cdot}{\cdot}$, and we will not distinguish between primal and dual vectors (so the distinction between normal and polar cones is likewise blurred).

Formally, we consider polyhedral domains written in \revise{standard} form as
\begin{equation}
\label{eq:polyhedron}
\points
	= \setdef{\point \in \R_{+}^{\nCoords}}{\mat \point = \cvec}
\end{equation}
for some matrix $\mat\in\R^{\nConstr\times\nCoords}$ and $\cvec\in\R^{\nCoords}$.%
\footnote{Inequality constraints of the form $\mat\point\leq\cvec$ can also be accommodated in \eqref{eq:polyhedron} by introducing the associated slack variables $s = \cvec - \mat\point \geq 0$.
Even though this leads to a more verbose presentation of $\points$, the form \eqref{eq:polyhedron} is much more convenient in terms of notational overhead, so we stick with the equality formulation throughout.}
Moreover, to avoid trivialities, we further assume that $\points$ admits a \emph{Slater point}, \ie there exists some $\point\in\points$
such that $\point_{\coord} > 0$ for all $\coord=1,\dotsc,\nCoords$.
This setup is particularly flexible, as it allows us to identify the \emph{active} constraints at $\point\in\points$ with the zero components of $\point$.

Elaborating further on this, since $\braket{\vecfield(\sol)}{\point-\sol} \geq 0$ for all $\point\in\points$ and any solution $\sol$ of \eqref{eq:VI}, we directly infer that $-\solvec$ is an element of the normal cone $\ncone(\sol)$ to $\points$ at $\sol$.
In our polyhedral setting, $\ncone(\sol)$ admits an especially simple representation as
\begin{equation}
\label{eq:ncone-sharp}
\ncone(\sol)
	= \row(\mat)
		- \setdef{(\slack_{1},\dotsc,\slack_{\nCoords})\in\R_{+}^{\nCoords}}{\slack_{\coord}=0 \text{ whenever } \sol_{\coord} = 0}
\end{equation}
where $\row(\mat) = (\ker\mat)^{\perp} \subseteq \R^{\nCoords}$ denotes the row space of $\mat$ \cite[Ex.~5.2.6]{HUL01}.
As a result, we see that $\sol$ is a solution of \eqref{eq:VI} if and only if $\vecfield(\sol)$ can be written in the form
\begin{equation}
\label{eq:slacks}
\solvec - \sum_{\coord\in\actcoords} \slack_{\coord}\bvec_{\coord}
	\in \row(\mat)
\end{equation}
for an ensemble of non-negative \emph{slackness coefficients} $\slack_{\coord} \geq 0$, $\coord\in\actcoords$, where
\begin{equation}
\label{eq:actcoords}
\actcoords
	\equiv \actcoords(\sol)
	= \setdef{\coord}{\sol_{\coord} = 0}
\end{equation}
denotes the set of inequality constraints of \eqref{eq:polyhedron} that are active at $\sol$.
\revise{With all of this in mind}, \revise{we distinguish} the following solution configurations (see also \cref{fig:simplex}).

\begin{definition}[Sharpness]
\label{def:sharp}
Let $\sol\in\points$ be a solution of \eqref{eq:VI} with associated slackness coefficients $\slack_{\coord}$, $\coord\in\actcoords$, as per \eqref{eq:slacks}.
The set of \emph{sharp} \textpar{$\sharp$} and \emph{flat} \textpar{$\flat$} directions at $\sol$ are respectively defined as
\begin{equation}
\label{eq:sharp}
\sharps
	= \setdef{\coord\in\actcoords}{\slack_{\coord} > 0}
	\qquad
	\text{and}
	\qquad
\flats
	= \setdef{\coord\in\actcoords}{\slack_{\coord} = 0},
\end{equation}
and we say that $\vecfield$ is \emph{sharp} at $\sol$ if $\sharps = \actcoords$ \textpar{or, equivalently, if $\flats = \varnothing$}.
The \emph{sharpness} of $\vecfield$ at $\sol$ is then defined as
\begin{equation}
\label{eq:drift}
\drift
	= \min\nolimits_{\coord\in\sharps} \slack_{\coord}
\end{equation}
and, finally, if $\sol$ is an extreme point of $\points$, we will say that $\sol$ is itself \emph{sharp}.
\end{definition}

The terminology ``sharp'' and ``flat'' alludes to the case where $\vecfield$ is a gradient field, and is best illustrated by an example.
To wit, let $\obj(\point_{1},\point_{2}) = \point_{1} + \tfrac{1}{2}(\point_{2}-1)^{2}$ for $\point_{1},\point_{2}\geq0$, so $\obj$ admits a (unique) global minimizer at $\sol = (0,1)$.
Applying \cref{def:sharp} to $\vecfield = \nabla\obj$, we readily get $\sharps = \{1\}$ and $\flats=\{2\}$, reflecting the fact that $\obj(\point_{1},1)$ exhibits a sharp minimum at $0$ along $\point_{1}$ whereas the landscape of $\obj(0,\point_{2})$ is flat to first-order around $1$ along $\point_{2}$.


\begin{figure}
\centering
%
%
\begin{tikzpicture}
[scale=.75,
>=stealth,
edgestyle/.style={-, line width=.75pt},
vecstyle/.style = {->, line width=.75pt},
nodestyle/.style={circle, fill=black,inner sep = .5pt},
plotstyle/.style={color=DarkGreen!80!Cyan,thick}]

\def\unit{1.25}
\def\costhirty{0.8660256}
\def\tanthirty{0.57735026919}
\def\upangle{60}
\def\downangle{0}
\def\conescale{1.75}
\small

\coordinate (O) at (0,-\unit/6);

\coordinate (P1) at (-2*\unit*\costhirty,-\unit) {};
\coordinate (P2) at (2*\unit*\costhirty,-\unit) {};
\coordinate (P3) at (0,2*\unit) {};

\coordinate (sol) at ($.5*(P3)+.5*(P1)$);
\coordinate (X) at (1.4*\unit,0.1*\unit);


\fill [PrimalFill] (P1) -- (P2) -- (P3) -- cycle;

\draw [edgestyle, thick, DualColor] (sol.center) -- ++(\upangle+90:\conescale*\unit) node [near end,above right] {$\ncone(\sol)$};

\vphantom{
\fill [DualColor!5] (P3)++(\upangle+90:\conescale*\unit) -- (P3) --++(\downangle+30:\conescale*\unit);
}

\node (X) at (X) [PrimalColor] {$\points$};

\draw [->, gray!50] (O) to (P1);
\draw [->, gray!50] (O) to (P2);
\draw [->, gray!50] (O) to (P3);


\draw [edgestyle,PrimalColor] (P1) -- (P2) -- (P3) -- (P1);

\node [nodestyle, label = right:$\sol$] (sol) at (sol) {.};
\draw [vecstyle] (sol.center) -- ($(sol)+(-1,\tanthirty)$) node [midway, below left, black] {$-\solvec$};

\end{tikzpicture}
\quad
%
%
\begin{tikzpicture}
[scale=.75,
>=stealth,
edgestyle/.style={-, line width=.75pt},
vecstyle/.style = {->, line width=.75pt},
nodestyle/.style={circle, fill=black,inner sep = .5pt},
plotstyle/.style={color=DarkGreen!80!Cyan,thick}]

\def\unit{1.25}
\def\costhirty{0.8660256}
\def\tanthirty{0.57735026919}
\def\upangle{60}
\def\downangle{0}
\def\conescale{1.75}
\small

\coordinate (O) at (0,-\unit/6);

\coordinate (P1) at (-2*\unit*\costhirty,-\unit) {};
\coordinate (P2) at (2*\unit*\costhirty,-\unit) {};
\coordinate (P3) at (0,2*\unit) {};

\coordinate (sol) at ($(P3)$);
\coordinate (X) at (1.4*\unit,0.1*\unit);


\fill [PrimalFill] (P1) -- (P2) -- (P3) -- cycle;

\draw [edgestyle, thick, DualColor] (sol.center) -- ++(\upangle+90:\conescale*\unit) node [midway,below left] {$\ncone(\sol)$};
\fill [DualColor!5] (sol)++(\upangle+90:\conescale*\unit) -- (sol) --++(\downangle+30:\conescale*\unit);
\draw [edgestyle, thick, DualColor] (sol.center) -- ++(\upangle+90:\conescale*\unit) node {};
\draw [edgestyle, thick, DualColor] (sol.center) -- ++(\downangle+30:\conescale*\unit);

\node (X) at (X) [PrimalColor] {$\points$};

\draw [->, gray!50] (O) to (P1);
\draw [->, gray!50] (O) to (P2);
\draw [->, gray!50] (O) to (P3);


\draw [edgestyle,PrimalColor] (P1) -- (P2) -- (P3) -- (P1);

\node [nodestyle, label = right:$\sol$] (sol) at (sol) {.};
\draw [vecstyle] (sol.center) -- ($(sol)+(-\unit,\tanthirty*\unit)$) node [very near end,right,black] {\;$-\solvec$};

\end{tikzpicture}
\quad
%
%
\begin{tikzpicture}
[scale=.75,
>=stealth,
edgestyle/.style={-, line width=.75pt},
vecstyle/.style = {->, line width=.75pt},
nodestyle/.style={circle, fill=black,inner sep = .5pt},
plotstyle/.style={color=DarkGreen!80!Cyan,thick}]

\def\unit{1.25}
\def\costhirty{0.8660256}
\def\tanthirty{0.57735026919}
\def\upangle{60}
\def\downangle{0}
\def\conescale{1.75}
\small

\coordinate (O) at (0,-\unit/6);

\coordinate (P1) at (-2*\unit*\costhirty,-\unit) {};
\coordinate (P2) at (2*\unit*\costhirty,-\unit) {};
\coordinate (P3) at (0,2*\unit) {};

\coordinate (sol) at ($(P3)$);
\coordinate (X) at (1.4*\unit,0.1*\unit);


\fill [PrimalFill] (P1) -- (P2) -- (P3) -- cycle;

\draw [edgestyle, thick, DualColor] (sol.center) -- ++(\upangle+90:\conescale*\unit) node [midway,below left] {$\ncone(\sol)$};
\fill [DualColor!5] (sol)++(\upangle+90:\conescale*\unit) -- (sol) --++(\downangle+30:\conescale*\unit);
\draw [edgestyle, thick, DualColor] (sol.center) -- ++(\upangle+90:\conescale*\unit) node {};
\draw [edgestyle, thick, DualColor] (sol.center) -- ++(\downangle+30:\conescale*\unit);

\node (X) at (X) [PrimalColor] {$\points$};

\draw [->, gray!50] (O) to (P1);
\draw [->, gray!50] (O) to (P2);
\draw [->, gray!50] (O) to (P3);


\draw [edgestyle,PrimalColor] (P1) -- (P2) -- (P3) -- (P1);

\node [nodestyle, label = right:$\sol$] (sol) at (sol) {.};
\draw [vecstyle] (sol.center) -- ($(sol)+(1/2,1)$) node [near end,left,black] {$-\solvec$};

\end{tikzpicture}
\hspace{2em}
\caption{%
\beginrev
Different boundary solution configurations on the $2$-dimensional unit simplex $\points = \setdef{(\point_{1},\point_{2},\point_{3})\in\R_{+}^{3}}{\point_{1} + \point_{2} + \point_{3} = 1}$ of $\R^{3}$:
a non-extreme solution where $\vecfield$ is sharp
($\actcoords = \sharps = \braces*{1}$, $\flats = \varnothing$; left),
an extreme solution where $\vecfield$ is not sharp
($\actcoords = \braces{1,2}$, $\sharps = \braces{1}$, $\flats=\braces*{2}$; middle),
and
a sharp solution ($\actcoords = \sharps = \braces*{1, 2}$, $\flats = \varnothing$; right).}
\label{fig:simplex}
\end{figure}


\subsection{Convergence rate analysis}
\label{sec:rate-sharp}

We are now in a position to state and prove our refinement of \cref{thm:general} for linearly constrained problems.
To that end, following \cite{ABB04}, we will assume in the rest of this section that \eqref{eq:BPM} is run with a Bregman regularizer $\hreg$ that is adapted to the polyhedral structure of $\points$ as per the definition below:

\begin{definition}
\label{def:decomposable}
Let $\points$ be a polyhedral domain of the form \eqref{eq:polyhedron}.
Then, a Bregman regularizer $\hreg$ on $\points$ is said to be \emph{decomposable with kernel $\hker$} if
\begin{equation}
\label{eq:decomposable}
\hreg(\point)
	= \sum_{\coord=1}^{\nCoords} \hker(\point_{\coord})
	\quad
	\text{for all $\point\in\points$}.
\end{equation}
for some continuous function $\hker\from\R_{+}\to\R$ such that
\begin{enumerate*}
[\upshape(\itshape i\hspace*{1pt}\upshape)]
\item
$\hker''(z)>0$ for all $z>0$;
and
\item
$\hker''$ is locally Lipschitz on $(0,\infty)$.
\end{enumerate*}
\end{definition}

In addition to facilitating calculations, the notion of decomposability will further allow us to describe the convergence rate of the iterates of \eqref{eq:BPM} near the boundary of $\points$ in finer detail.
In fact, as it turns out, the speed of convergence along a given direction will actually be determined by the behavior of the derivative of the Bregman kernel $\hker$ near $0$.

In this regard, there are two distinct regimes to consider.
First, if $\lim_{\point\to0^{+}}\hker'(\point) = -\infty$, it is straightforward to see that $\dom\subd\hreg = \relint\points$ so, by \refinapp{lem:mirror}{app:aux}, the iterates $\curr$ of \eqref{eq:BPM} will remain in $\relint\points$ for all $\run$;
in this case $\hreg$ is essentially smooth \textendash\ or \emph{Legendre} \textendash\ in the sense of \cite[Chap.~26]{Roc70}, and we will refer to it as \emph{steep}.
Otherwise, if $\hker'(0)$ exists and is finite, $\curr$ may reach the boundary of $\points$ in a finite number of iterations;
we will refer to this case as \emph{non-steep}.
The key difference between these two regimes is that, in the non-steep case, the algorithm may achieve convergence in a finite number of steps (at least along certain directions).
On the other hand, even though finite-time convergence is not possible in the steep regime, the algorithm's rate of convergence may still depend on the boundary behavior of $\hker$.
To illustrate this, we will consider the following concrete cases:

\begin{assumption}
\label{asm:ker}
If $\hker\from\R_{+}\to\R$ is a kernel function as per \cref{def:decomposable},
\revise{$\hker'$ exhibits one of the following behaviors as $\point\to0^{+}$:}
\begin{enumerate}
[\upshape(\itshape a\hspace*{.5pt}\upshape)]
\item
\label[case]{asm:ker-Eucl}
\emph{Euclidean-like:}
	\tabto{7em}
	$\liminf_{\point\to0^{+}} \hker'(\point) > -\infty$.
\item
\label[case]{asm:ker-log}
\emph{Entropy-like:}
	\tabto{7em}
    \revise{$\liminf_{\point\to0^{+}} \bracks{\hker'(\point) - \log\point} > -\infty$.}
\item
\label[case]{asm:ker-power}
\emph{Power-like:}
	\tabto{7em}
	$\liminf_{\point\to0^{+}} \point^{\kernelexp} \hker'(\point) > -\infty$ for some $\kernelexp\in(0,1)$.
\end{enumerate}
\end{assumption}

\begin{remark*}
\Crefrange{asm:ker-Eucl}{asm:ker-power} above respectively mean that $\abs{\hker'(\point)}$ grows as $\bigoh(1)$, $\bigoh(\abs{\log\point})$ or $\bigoh(1/\point^{\kernelexp})$ as $\point\to0^{+}$.
Clearly, we have \ref{asm:ker-Eucl}$\implies$\ref{asm:ker-log}$\implies$\ref{asm:ker-power} so these cases are not exclusive;
nonetheless, to avoid overloading the presentation, when we assume \ref{asm:ker-log}, we will tacitly imply that \labelcref{asm:ker-Eucl} does not also hold at the same time \textendash\ and likewise for \ref{asm:ker-power}.
\hfill
\endenv
\end{remark*}

With all this in hand, we proceed below to show that, in linearly constrained problems, \eqref{eq:BPM} converges along sharp directions at $\sol$ at an accelerated rate relative to \cref{thm:general}:
sublinear rates may become linear, and linear rates transform to convergence in finite time.

\begin{theorem}
\label{thm:sharp}
Suppose that \eqref{eq:BPM} is run in a polyhedral domain with a decomposable regularizer as per \cref{def:decomposable}.
Suppose further that \cref{asm:Lipschitz,asm:strong,asm:signal-base,asm:ker} hold, and that the method's step-size and initialization satisfy the requirements of \cref{thm:general}.
Then, for all $\coord\in\sharps$, we have:
\smallskip
\begin{subequations}
\label{eq:rate-sharp}
\begin{enumerate}
[\upshape(\itshape a\hspace*{.5pt}\upshape)]
\item
Under \cref{asm:ker}\ref{asm:ker-Eucl}:
there exists some $\nRuns \geq \start$ such that
\begin{align}
\label{eq:rate-sharp-Eucl}
\state_{\coord,\run}
	&= 0
	\quad
	\text{for all $\run\geq\nRuns$}
\shortintertext{%
\item
Under \cref{asm:ker}\ref{asm:ker-log}:}
\label{eq:rate-sharp-log}
\state_{\coord,\run}
	&= \bigoh\parens[\big]{\exp(-\step\drifteff\run/2)}
\shortintertext{%
\item
Under \cref{asm:ker}\ref{asm:ker-power}:}
\label{eq:rate-sharp-power}
\state_{\coord,\run}
	&= \bigoh\parens[\big]{(\step\drifteff\run/2)^{-1/\kernelexp}}
\end{align}
\end{enumerate}
\end{subequations}
where
\begin{equation}
\label{eq:drift-eff}
\drifteff
	= \begin{cases}
		\drift
			&\quad
			\text{if $\vecfield$ is sharp at $\sol$ \textpar{\ie $\flats=\varnothing$}},
			\\
		\drift/\varrho
			&\quad
			\text{otherwise},
	\end{cases}
\end{equation}
and $\varrho \equiv \varrho(\mat,\cvec,\sol) \geq 1$ is a positive constant that depends only on $\points$ and $\sol$.
\end{theorem}

In particular, if $\sol$ is sharp, we have the following immediate corollary of \cref{thm:sharp}.

\begin{corollary}
\label{cor:sharp}
If $\sol$ is sharp, then $\norm{\curr - \sol} \to 0$ at a rate given by
\eqref{eq:rate-sharp-Eucl}, \eqref{eq:rate-sharp-log}, or \eqref{eq:rate-sharp-power},
depending respectively on whether \cref{asm:ker-Eucl}, \labelcref{asm:ker-log}, or \labelcref{asm:ker-power} of \cref{asm:ker} holds.
\end{corollary}

\begin{proof}
First, note that $\sol$ is sharp if and only if $\lspan\setdef{\bvec_{\coord}}{\coord\in\sharps} + \row(\mat) = \R^{\nCoords}$:
indeed, since $\points$ is a polyhedron, $\sol$ is extreme if and only if $\ncone(\sol)$ has nonempty topological interior, and this, combined with \eqref{eq:ncone-sharp} and the fact that $\sharps = \actcoords$ (since $\vecfield$ is sharp at $\sol$), proves our assertion.
We thus conclude that, for all $\coord = 1,\dotsc,\nCoords$, there exist $\coef_{\coord\coordalt} \in \R$, $\coordalt\in\sharps$, such that $\bvec_{\coord} - \sum_{\coordalt\in\sharps} \coef_{\coord\coordalt} \bvec_{\coordalt} \in \row(\mat)$,
and hence, for all $\run = \running$, we have $\state_{\coord,\run} - \sol[\state_{\coord}] = \sum_{\coordalt \in \sharps} \coef_{\coord\coordalt} \state_{\coordalt,\run}$.
Our claim then follows from \cref{thm:sharp} and the fact that all norms are equivalent on $\R^{\nCoords}$.
\end{proof}

To facilitate comparisons with the non-sharp regime, the guarantees of \cref{thm:sharp} are juxtaposed with those of \cref{thm:general} in \Cref{tab:rates-sharp}.
Beyond this comparison, \cref{thm:sharp} is the main result of this section so, before proving it, some further remarks are in order.

\setcounter{remark}{0}

\begin{table}[tbp]
\centering
\footnotesize
\renewcommand{\arraystretch}{1.25}

\begin{tabular}{lcccc}
\toprule
	&\textbf{Bregman Kernel ($\hker$)}
    &\revise{\bfseries Bregman rate (Thm.~\ref{thm:general})}
    &\revise{\bfseries Sharp rate (Thm.~\ref{thm:sharp})}
	\\
\midrule
\scshape{Euclidean}	
	&$\point^{2}/2$
    &Linear
	&Finite time
	\\
\scshape{Entropic}
	&$\point\log\point$
	&$\bigoh(1/\run)$
	&Linear 
	\\
\scshape{Tsallis}
	&$[\qexp(1-\qexp)]^{-1} (\point - \point^{\qexp})$
	&$\bigoh(1/\run^{\qexp/(2-\qexp)})$
	&$\bigoh(1/\run^{1/(1-\qexp)})$
	\\
\scshape{Hellinger}
	&$-\sqrt{1-\point^{2}}$
    &$\bigoh(1/\run^{1/3})$
	&$\bigoh(1/\run^{2})$
	\\
\bottomrule
\end{tabular}

\smallskip
\caption{Summary of the accelerated rates of convergence observed along sharp directions 
depending on the Bregman kernel (\cf \cref{def:decomposable}).
The Euclidean, entropic and Tsallis kernels are the prototypical examples of \cref{asm:ker-Eucl,asm:ker-log,asm:ker-power} of \cref{asm:ker}.
\revise{The Bregman rate column refers to the Bregman divergence $\breg(\sol,\curr)$;
the sharp rate refers to $\state_{\coord,\run}$ for a sharp direction $\coord\in\sharps$ at $\sol$.}}
\label{tab:rates-sharp}
\end{table}


\begin{remark}
[Solution configurations]
By construction (and the fact that $\points$ admits a Slater point), it is straightforward to verify that $\vecfield$ is sharp at $\sol$ if and only if $\vecfield(\sol) \in -\relint(\ncone(\sol))$;
likewise, $\sol$ is itself sharp if and only if $\vecfield(\sol) \in -\intr(\ncone(\sol))$.
As we noted in the proof of \cref{cor:sharp}, the latter condition is equivalent to asking that $\lspan\setdef{\bvec_{\coord}}{\coord\in\sharps} + \row(\mat) = \R^{\nCoords}$, a condition which describes precisely the informal requirement that the sharp directions at $\sol$ suffice to characterize it.
By contrast, if $\vecfield$ is sharp at some non-extreme point $\sol$, there exists some (nonzero) $\tvec\in\tcone(\sol)$ such that $\braket{\vecfield(\sol)}{\tvec} = 0$, indicating that the accelerated rates of \cref{thm:sharp} cannot be active along the residual direction $\pvec$.
We illustrate these distinct solution configurations in \cref{fig:simplex}.
\hfill
\endenv
\end{remark}
\begin{remark}[Tightness and the structure of $\points$]
We underline 
that the dependence of $\drifteff$ on the structure of $\points$ in the second branch of \eqref{eq:drift-eff} cannot be lifted.
To see this, let
\begin{equation}
\points
	= \setdef{\point\in\R_{+}^{2}}{\point_{1} = \eps\point_{2}}
\end{equation}
\ie $\mat = \bracks{1 \; -\eps}$ and $\ker\mat$ is spanned by the vector $\pvec = (\eps,1)$.
Then, if we take $\vecfield(\point) = \point - \base$ with $\base_{1} \leq 0$ and $\base_{2} \geq 0$ (so that the origin is a solution), and we equip $\points$ with the Bregman regularizer induced by the entropic kernel $\hker(\point) = \point\log\point$, a straightforward calculation shows that the iterates of \eqref{eq:MD} satisfy the recursion
\begin{equation}
\eps \hker'(\state_{1,\run+1}) + \hker'(\state_{2,\run+1})
	= \bracks[\big]{\eps \hker'(\state_{1,\run}) + \hker'(\state_{2,\run})}
		+ \step \parens[\big]{\eps\state_{1,\run} + \state_{2,\run}}
		- \step \braket{\base}{\pvec}.
\end{equation}
Thus, letting $\curr[\chi] = \state_{2,\run} = \state_{1,\run} / \eps$, the above can be rewritten as
\begin{equation}
\eps \hker'(\eps\next[\chi]) + \hker'(\next[\chi])
	= \bracks[\big]{\eps \hker'(\eps\curr[\chi]) + \hker'(\curr[\chi])}
		+ \step (\eps^{2} + 1) \curr[\chi]
		- \step \braket{\base}{\pvec}
\end{equation}
and hence, with $\hker'(\point) = 1 + \log\point$, we finally get
\begin{equation}
\next[\chi]
	= \curr[\chi] \exp\parens*{
			- \step \frac{\eps^{2} + 1}{\eps + 1} \curr[\chi]
			+ \step \frac{\braket{\base}{\pvec}}{\eps + 1}
		}.
\end{equation}
Now, if $\braket{\base}{\pvec} < 0$,
we readily infer that $\curr[\chi] = \state_{2,\run}$ converges linearly to $0$
at a rate of
\(
\curr[\chi]
	\sim \exp\parens*{- \abs{\braket{\base}{\pvec}}/(1+\eps) \cdot \step \run}
\)
as predicted by \cref{thm:sharp}.
In particular, if $\vecfield(0) = - \base = (\drift, \drift)$
with $\drift > 0$, the iterates of \eqref{eq:MD} converge geometrically to zero with exponent $\step \drift$, which matches the estimate of \cref{thm:sharp} up to a factor of $1/2$ in the exponent.
On the contrary, if $\vecfield(0) = -\base = (\drift, 0)$, the term $\abs{\braket{-\base}{\pvec}} = \eps\drift$ depends on the linear structure of $\points$ and can be arbitrarily bad as $\eps$ goes to zero.
This illustrates why one cannot do away with the dependence on the linear structure of $\points$ when $\vecfield$ is not sharp at $\sol$.
\hfill
\endenv
\end{remark}

\subsection{Proof of \cref{thm:sharp}}
\label{sec:proof-sharp}

We now proceed to the proof of \cref{thm:sharp}, beginning with two helper lemmas tailored to the polyhedral structure of $\points$.
The first is a book-keeping result regarding the subdifferentiability of $\hreg$.

\begin{lemma}
\label{lem:dh}
Let $\hreg$ be a decomposable regularizer on $\points$ with kernel $\hker$ as per \cref{def:decomposable}.
Then the domain of subdifferentiability of $\hreg$ is $\proxdom = \setdef{\point \in \points}{\point_{\coord} \in \dom \subd \hker \text{ for all } \coord = 1,\dots,\nCoords}$ and a continuous selection of $\subd\hreg$ is given by the expression
\begin{equation}
\label{eq:dh}
\nabla \hreg(\point)
	= \sum_{\coord = 1}^{\nCoords} \hker'(\point_{\coord})\bvec_{\coord}
	\quad
	\text{for all $\point\in\proxdom$}.
\end{equation}
\end{lemma}

\begin{proof}
See \cite[Thm~23.8]{Roc70}, whose conditions are satisfied because $\points$ is polyhedral.
\end{proof}

The second ingredient we will need is a separation result in the spirit of Farkas' lemma.

\begin{restatable}{lemma}{separation}
\label{lem:separation}
Let $\points$ be a polyhedral domain of the 
form \eqref{eq:polyhedron}.
Then, for all $\sol\in\points$, there exists $\polycst = \polycst(\mat, \cvec, \sol) \geq 1$ such that, for all $\coords \subseteq \actcoords \equiv \actcoords(\sol)$, at least one of the following holds:
\begin{enumerate}
[(\itshape a\upshape)]
\item
$\coords \neq \varnothing$ and there exists $\coord \in \actcoords\setminus\coords$ such that
\(
\point_{\coord}
	\leq \polycst \max\setdef{\point_{\coordalt}}{\coordalt \in \coords}
\)
for all $\point\in\points$.
\label[case]{itm:inactive}
\item
There exists $\pvec \in \ker\mat$ such that $\norm{\pvec} \leq \polycst$, $\pvec_{\coord} = 0$ if $\coord \in \coords$ and $\polycst \geq \pvec_{\coord} \geq 1$ if $\coord \in \actcoords \setminus \coords$.
\label[case]{itm:active}
\end{enumerate}
\end{restatable}

The proof of \cref{lem:separation} is based on Farkas' lemma so we relegate it \refapp[to]{app:aux}.
Armed with all this, we can finally proceed to prove our result for linearly constrained problems.

\begin{proof}[Proof of \cref{thm:sharp}]
We will consider two main cases, namely $\lim_{\point\to0^{+}}\hreg'(\point) = -\infty$ (the steep case) and $\lim_{\point\to0^{+}}\hreg'(\point) > -\infty$ (the non-steep case).
The steep regime will cover \cref{asm:ker-log,asm:ker-power} of \cref{asm:ker}, whereas the non-steep regime will account for \cref{asm:ker-Eucl}.

\para{Case 1: the steep regime}
Note first that, without loss of generality, \cref{asm:ker-log,asm:ker-power} respectively imply that there exist $\kernelcst\in\R$ and $\thres>0$ such that, for all $\point\in(0,\thres)$, we have:
\begin{subequations}
\label{eq:ker-growth}
\begin{flalign}
\qquad
\text{Under \cref{asm:ker}\ref{asm:ker-log}:}
	&\;\;\;
	\text{$\hker'(\point) \geq \log\point - \kernelcst$}
	&&
	\\
\text{Under \cref{asm:ker}\ref{asm:ker-power}:}
	&\;\;\;
	\text{$\hker'(\point) \geq -\kernelcst\point^{-\kernelexp}$}
	&&
\end{flalign}
\end{subequations}
With this in mind, let $\radius > 0$ be sufficiently small so that $\ball_{\radius} \defeq \setdef{\point\in\points}{\norm{\point - \sol} \leq \radius}$ satisfies:
\begin{itemize}
\item
$\ball_{\radius} \subseteq \nhd \cap \basin$ with $\basin$ and $\nhd$ defined by \eqref{eq:strong} and \eqref{eq:Breg-upper} respectively.
\item
If $\point \in \ball_{\radius}$ then $\point_{\coord} < \thres$ for all $\coord \in \actcoords$.
\item
If $\point \in \ball_{\radius}$, then $\dnorm{\vecfield(\point) - \vecfield(\sol)} \leq \drift/(2\polycst)$ with $\drift$ 
by \eqref{eq:drift} and $\polycst$ 
by \cref{lem:separation}.
\end{itemize}

Now, recall that Step 1 of the proof of \cref{thm:general} implies that the iterates of \eqref{eq:BPM} will remain in $\ball_{\radius}$ for all $\run$ if $\init$ is initialized sufficiently close to~$\sol$.
Given this stability guarantee, we will construct below two sets $\coords_{\sharp} \subseteq \sharps$, $\coords_{\flat} \subseteq \flats$ such that, for all $\coord \in \coords \defeq \coords_{\sharp} \cup \coords_{\flat}$, we have
\begin{align}
\label{eq:propCoord}
\state_{\coord,\run}
	\leq \polycst^{\abs{\coords}}\!\cdot\!\begin{cases}
		\exp\parens*{\kernelcst + \polycst \dnorm{\grad \hreg(\init)}\!+\!\polycst \bregbdedcst\!-\!\step \drifteff(\run - 1)/2}
			&
			\text{under \cref{asm:ker}\ref{asm:ker-log},}
			\\
		\kernelcst \pospart*{\step\drifteff(\run - 1)/2 - \polycst \dnorm{\grad \hreg(\init)} - \polycst \bregbdedcst}^{-\frac{1}{\kernelexp}}
			&
			\text{under \cref{asm:ker}\ref{asm:ker-power}},
		\end{cases}
\end{align}
where
$\bregbdedcst \defeq \sup_{\point \in \ball_{\radius}} \dnorm*{ \sum_{\coord \notin \actcoords} \hker'(\point_{\coord}) \bvec_\coord}$, $\polycst \geq 1$ is the constant given by \cref{lem:separation} and
$\drifteff$ is defined as in \eqref{eq:drift-eff} with $\varrho = \polycst\abs{\actcoords}$.

Our construction proceeds inductively, starting with $\coords_{\sharp} = \coords_{\flat} = \varnothing$, for which the stated property holds trivially.
For the inductive step, if \eqref{eq:propCoord} holds for $\coords_{\sharp} \subsetneq \sharps$ and $\coords_{\flat} \subseteq \flats$, we will show that there exists some $\coordalt \in \actcoords \setminus \coords$ such that \eqref{eq:propCoord} still holds for $\coords \cup \{\coordalt\}$.
Since the number of active constraints is finite, the progressive addition of these indices will allow us to reach $\coords_{\sharp} = \sharps$, thus proving our initial claim.

To carry all this out, assume that $\coords_{\sharp} \subsetneq \sharps$, $\coords_{\flat} \subseteq \flats$, and apply \cref{lem:separation} to $\coords = \coords_{\sharp} \cup \coords_{\flat}$.
If the first case of \cref{lem:separation} holds, then $\coords \neq \varnothing$ and there exists $\coord \in \actcoords$ such that
\begin{equation}
\state_{\coord,\run}
	\leq \polycst \max\nolimits_{\coordalt\in\coords} \state_{\coordalt,\run}
\end{equation}
so \eqref{eq:propCoord} still holds when $\coord$ is appended to $\coords_{\sharp}$ or $\coords_{\flat}$. 
Otherwise, the second case of \cref{lem:separation} holds and there exists $\pvec \in \ker\mat$ with $\norm{\pvec} \leq \polycst$,
$\pvec_{\coord} = 0$ if $\coord \in \coords$,
and
$\polycst \geq \pvec_{\coord} \geq 1$ if $\coord \in \actcoords \setminus \coords$.
Since $\hreg$ is steep, this means that $\next = \proxof{\curr}{-\step \lead[\signal]}$ belongs to $\proxdom = \relint \points$ for all $\run=\running$, so 
$\ncone(\next)$ 
is the affine hull of $\points$, \ie $\ncone(\next) = \row(\mat)$.
Hence, \refinapp{lem:mirror}{app:aux} guarantees
\begin{equation}
\label{eq:proof-sharp-mirror-step}
\subsel\hreg(\next)
	- \subsel\hreg(\curr)
	+ \step \lead[\signal]
	\in \row(\mat)
\end{equation}
so, telescoping from $\runalt = \start$ to $\run-1$, we get
\(
\subsel\hreg(\curr)
	- \subsel\hreg(\init)
	+ \step \sum_{\runalt = \start}^{\run-1} \iterlead[\signal]
	\in \row(\mat).
\)
Thus, taking the scalar product with $\pvec \in \ker\mat$ yields
\begin{equation}
\sum_{\coord=1}^{\nCoords} \hker'(\state_{\coord,\run}) \pvec_{\coord}
	= \braket{\subsel\hreg(\init)}{\pvec}
		- \step \sum_{\runalt=\start}^{\run - 1} \braket{\iterlead[\signal]}{\pvec}
\end{equation}
so, after rearranging and invoking \eqref{eq:slacks} to write $\braket{\vecfield(\sol)}{\pvec} = \sum_{\coord \in \actcoords} \slack_{\coord} \pvec_{\coord}$, we get
\begin{align}
\sum_{\coord \in \actcoords \setminus \coords} \hker'(\state_{\coord,\run}) \pvec_{\coord}
	+ \sum_{\coord \in \coords} \hker'(\state_{\coord,\run}) \pvec_{\coord}
	&= \braket{\subsel\hreg(\init)}{\pvec}
		- \step \sum_{\runalt=\start}^{\run - 1} \sum_{\coord \in \actcoords} \slack_{\coord} \pvec_{\coord}
	\notag\\
	&+\step \sum_{\runalt = \start}^{\run - 1} \braket{\vecfield(\sol) - \iterlead[\signal]}{\pvec}
		- \sum_{\coord \notin \actcoords} \hker'(\state_{\coord,\run}) \pvec_{\coord}.
\end{align}
Finally, by the properties we used to construct $\pvec$, we further have
\begin{align}
\label{eq:proof-sharp-template-ineq}
\sum_{\mathclap{\coord\in\actcoords \setminus \coords}} \hker'(\state_{\coord,\run}) \pvec_{\coord}
	& \leq \polycst \dnorm{\subsel\hreg(\init)}
		- \step\parens{\run-1} \sum_{\mathclap{\coord \in \actcoords \setminus \coords}} \slack_{\coord} \pvec_{\coord}
	\notag\\
	&\qquad
		+ \step\sum_{\runalt=\start}^{\run - 1} \polycst \dnorm{\vecfield(\sol) - \iterlead[\signal]}
		+ \polycst \dnorm*{\sum_{\coord \notin \actcoords} \hker'(\state_{\coord,\run}) \bvec_\coord}
	\notag\\
	& \leq \polycst \dnorm{\subsel\hreg(\init)}
		- \step (\run - 1) \sum_{\mathclap{\coord \in \actcoords \setminus \coords}} \slack_{\coord} \pvec_{\coord}
		+ \tfrac{1}{2} \step \drift (\run - 1)
		+ \polycst \bregbdedcst.
\end{align}
where the second inequality follows from how we chose $\ball_{\radius}$ at the beginning of the proof.

We conclude by distinguishing whether $\vecfield$ is sharp at $\sol$ (\ie if $\flats = \varnothing$ or not).
\begin{enumerate}
[left=0em,label={\bfseries Case \arabic*:}]
\item
If $\flats=\varnothing$, we have $\actcoords \setminus \coords = \sharps \setminus \coords_{\sharp}$ and $\slack_{\coord} \geq \drift$ for all $\coord \in \actcoords \setminus \coords$, so \eqref{eq:proof-sharp-template-ineq} gives
\begin{equation}
\sum_{\coord \in \actcoords \setminus \coords} \hker'(\state_{\coord,\run}) \pvec_{\coord} + \step (\run - 1) \drift \pvec_{\coord}
	\leq \polycst \dnorm{\subsel\hreg(\init)}
		+ \tfrac{1}{2} \step\drift(\run-1)
		+ \polycst \bregbdedcst.
\end{equation}
Choosing the coordinate $\coordalt \in \actcoords \setminus \coords$ corresponding to the smallest term in the sum on the \ac{LHS}, we get
\begin{align}
	\parens*{\hker'(\state_{\coordalt,\run}) + \step (\run - 1) \drift} (\abs{\actcoords} \setminus \coords) \pvec_{\coordalt}
	\leq
	\polycst \dnorm{\subsel\hreg(\init)}
	+ \tfrac{1}{2} \step\drift(\run - 1)
	+ \polycst \bregbdedcst
\end{align}
and noting that $\abs{\actcoords\setminus \coords} \pvec_{\coordalt} \geq 1$ yields
\begin{align}
 \label{eq:proof-sharp-final-ineq1}
	\hker'(\state_{\coordalt,\run})
	\leq
	\polycst \dnorm{\subsel\hreg(\init)}
	-\tfrac{1}{2} \step\drift(\run - 1)
	+ \polycst \bregbdedcst\,.
\end{align}

\item
If $\flats \neq \varnothing$, then, since $\coords_{\sharp} \subsetneq \sharps$, the intersection of $\actcoords \setminus \coords$ and $\sharps$ is not empty so that, $\sum_{\coord \in \actcoords \setminus \coords} \slack_{\coord} \pvec_{\coord} \geq \sum_{\coord \in \actcoords \setminus \coords} \slack_{\coord} \geq \drift$ and the inequality \eqref{eq:proof-sharp-template-ineq} above becomes,
\begin{align}
	\sum_{\coord \in \actcoords \setminus \coords} \hker'(\state_{\coord,\run}) \pvec_{\coord}
	\leq
	\polycst \dnorm{\subsel\hreg(\init)}
	- \tfrac{1}{2} \step\drift(\run - 1)
	+ \polycst \bregbdedcst\,.
\end{align}

Now, choosing $\coordalt$ to be the coordinate in $\actcoords \setminus \coords$ which minimizes the \ac{LHS} and bounding $\abs{\actcoords \setminus \coords}$ by $1$ and $\abs{\actcoords}$, we get that
\begin{align}
\hker'(\state_{\coordalt,\run}) \pvec_{\coordalt}
	\leq
	\polycst \dnorm{\subsel\hreg(\init)}
	-\frac{\step \drift}{2\abs{\actcoords}} (\run - 1)
	+ \polycst \bregbdedcst\,.
\end{align}
Dividing both sides by $\pvec_{\coordalt}$ and using that it lies between $1$ and $\polycst$ gives
\begin{align}
 \label{eq:proof-sharp-final-ineq2}
	\hker'(\state_{\coordalt,\run})
	\leq
	\polycst \dnorm{\subsel\hreg(\init)}
	-\frac{\step \drift}{2 \polycst \abs{\actcoords}} (\run - 1)
	+ \polycst \bregbdedcst\,
\end{align}
\end{enumerate}
Thus, in both cases, there exists $\coordalt \in \actcoords \setminus \coords$ such that \eqref{eq:proof-sharp-final-ineq2} holds, since $\polycst \abs{\actcoords} \geq 1$.
Therefore, combining this inequality with \eqref{eq:ker-growth} we conclude that \eqref{eq:propCoord} holds for $\coordalt$,
and 
since $\polycst \geq 1$, we conclude that we can augment $\coords_{\sharp}$ or $\coords_{\flat}$ by $\coordalt$, depending on whether it belongs to $\sharps$ or $\flats$.

\para{Case 2: the non-steep regime}
The proof borrows the structure of the first case, though it is more direct. 
	Take $\radius > 0$ small enough such that $\ball_{\radius} \defeq \{ \point\in\points : \norm{\point - \sol} \leq \radius \}$ satisfies
\begin{itemize}
		\item $\ball_{\radius}$ is included in $\nhd \cap \basin$, and if $\point \in \ball_{\radius}$, then $\dnorm{\vecfield(\point) - \vecfield(\sol)} \leq \frac{\drift}{3 \polycst}$,
		\item if $\point, \pointalt \in \ball_{\radius}$ then, $\dnorm{\subsel\hreg (\point) - \subsel\hreg(\pointalt)} \leq \frac{\step \drift}{3 \polycst}$ where $\polycst$ is given by \cref{lem:separation}. This is possible since $\subsel\hreg$ is continuous at $\sol$,
 \item No other constraint $\point_{\coord} = 0$ with $\coord \notin \actcoords$ becomes active in $\ball_{\radius}$.
\end{itemize}
	As we have seen in 
 the proof of \cref{thm:general}, if \eqref{eq:BPM} is initialized close enough to~$\sol$, then all the iterates $\curr$ and the half-iterates $\lead$ for $\runalt=\running$ are contained in $\ball_{\radius}$.

 As above, fix some $\run \geq \nRuns$.
 We will build sets $\coords_{\sharp} \subseteq \sharps$, $\coords_{\flat} \subseteq \flats$ with the property that that
 \begin{align}
 \label{eq:propertynonsteep}
 \text{ for all } \coord \in \coords_{\sharp} \cup \coords_{\flat}, \state_{\coord,\run} = 0 \,.
 \end{align}

	Starting with $\coords_{\sharp} = \coords_{\flat} = \varnothing$, \cref{eq:propertynonsteep} is trivially verified.
	Now, take $\coords_{\sharp} \subsetneq \sharps$, $\coords_{\flat} \subseteq \flats$ which satisfy the desired property and, as before, apply \cref{lem:separation} with $\coords \defeq \coords_{\sharp} \cup \coords_{\flat}$.
If the first case of \cref{lem:separation} holds, then $\coords \neq \varnothing$ and there exists $\coordalt \in \actcoords$ such that,
\begin{align}
\state_{\coordalt,\run} & \leq \polycst \max\parens*{\state_{\coord,\run} : \coord \in \coords}
\end{align}
which yields the result by adding $\coord$ to $\coords_{\sharp}$ or $\coords_{\flat}$ depending whether it belongs to $\sharps$ or $\flats$.
Otherwise, if the second case holds, there is some $\pvec \in \ker\mat$ such that $\norm{\pvec} \leq \polycst$, $\pvec_{\coord} = 0$ if $\coord \in \coords$ and $\polycst \geq \pvec_{\coord} \geq 1$ if $\coord \in \actcoords \setminus \coords$.
For the sake of contradiction, assume that for all $\coord \in \actcoords \setminus \coords$, $\state_{\coord,\run} > 0$. Showing that this results in a contradiction will give us an additional coordinate $\coordalt \in \actcoords \setminus \coords$ for which $\state_{\coordalt,\run} = 0$ that we will then add to $\coords_{\sharp}$ or $\coords_{\flat}$ as in the first case.

Now, let us determine the normal cone at $\curr$. 
 Since $\curr$ belongs to $\ball_\radius^{\points}(\sol)$, no other constraint other than the ones 
 of $\coords$ can become active, and these constraints are actually active by the definition of $\coords$ and \cref{eq:propertynonsteep}. Hence, the normal cone at $\curr$ (see \cref{eq:ncone-sharp}) becomes
\begin{align}
	\ncone(\curr) &= \left\{-\sum_{\coord \in \coords} \slack_{\coord} \bvec_{\coord} : (\slack_{\coord})_{\coord \in \coords} \in (\R_+)^{\coords}\right\} + \row(\mat)\,.
\end{align}
Taking a scalar product between the last inclusion of \refinapp{lem:mirror}{app:aux} and $\pvec$, we get that
\begin{align}
	\braket*{ \subsel\hreg(\curr) - \subsel\hreg(\prev) - \step \beforelead[\signal]}{\pvec} = 0\,.
\end{align}
This means, from the definition of $\ball_\radius^{\points}(\sol)$, that, 
\begin{align}
	\step\braket{\vecfield(\sol)}{\pvec} &=
	\braket*{ \subsel\hreg(\curr) - \subsel\hreg(\prev)}{\pvec} + \step\braket{\vecfield(\sol) - \step \beforelead[\signal]}{\pvec}
	\leq
	\frac{2\step \drift}{3}\label{eq:proof-sharp-eucl-final-ineq}\,
\end{align}
However, by \eqref{eq:slacks} and the properties of $\pvec$, we also have 
\begin{align}
	\braket{\vecfield(\sol)}{\pvec} &=
	\sum_{\coord \in \actcoords \setminus \coords} \slack_{\coord} \pvec_{\coord} \geq \drift\,
\end{align}
which is in contradiction with \eqref{eq:proof-sharp-eucl-final-ineq}.
We may therefore iteratively add coordinates of $\actcoords$ for which $ \state_{\coord,\run} = 0$, which completes the induction and our proof.
\end{proof}

\section{Concluding remarks}
\label{sec:discussion}

Our results indicate that Euclidean regularization leads to faster trajectory convergence rates near \ac{SOS} solutions.
While this does not contradict the analysis of \cite{Nem04} \textendash\ which concerns the method's ergodic average and advocates the use of non-Euclidean regularizers in domains with a favorable geometry \textendash\ it \emph{does} run contrary to its spirit.
We attribute the source of this discrepancy
to the fact that Lipschitz continuity and second-order sufficiency are both norm-based conditions, so it is plausible to expect that norm-based regularizers would lead to better results.
This raises the question of what the corresponding rate analysis would give in the case of Bregman-based variants of \eqref{eq:Lipschitz} and \eqref{eq:strong}, \eg as in the recent works of \cite{BDX11,BBT17,LFN18,ABM19,ABM20,AM21,ABM21}.
We defer this analysis to future work.

\section*{Acknowledgments}
%
%
This research was supported in part by 
the French National Research Agency (ANR) in the framework of
the PEPR IA FOUNDRY project (ANR-23-PEIA-0003),
the ``Investissements d'avenir'' program (ANR-15-IDEX-02),
the LabEx PERSYVAL (ANR-11-LABX-0025-01),
MIAI@Grenoble Alpes (ANR-19-P3IA-0003).
PM is also a member of the Archimedes Research Unit/Athena RC\textendash NKUA, and was partially supported by project MIS 5154714 of the National Recovery and Resilience Plan Greece 2.0 funded by the European Union under the NextGenerationEU Program.

\appendix
\numberwithin{equation}{section}	
\numberwithin{lemma}{section}	
\numberwithin{proposition}{section}	
\numberwithin{theorem}{section}	
\numberwithin{corollary}{section}	

\section{Auxiliary results}
\label{app:aux}
%
%
We provide here a series of basic properties, helper lemmas and auxiliary results that we use repeatedly in our paper.

\subsection{Lemmas on numerical sequences}
The first two results 
concern numerical sequences.

\begin{lemma}
\label{lem:Polyak}
Consider two sequences of 
real numbers $\curr[\seq], \curr[\diff] \geq 0$, $\run=\running$, such that
\begin{equation}
\next[\seq]
	\leq \curr[\seq]
		- \curr[\diff] \curr[\seq]^{1+\rexp}
	\quad
	\text{for some $\rexp>0$ and all $\run=\running$}
\end{equation}
Then, for all $\run=\running$, we have:%
\begin{equation}
\next[\seq]
	\leq \frac{\init[\seq]}{\parens*{1 + \rexp \init[\seq]^{\rexp} \sum_{\runalt=\start}^{\run} \iter[\diff]}^{1/\rexp}}.
\end{equation}
\end{lemma}

\begin{proof}
See \cite[p.~46, Lem.~6]{Pol87}.
\end{proof}

The second result that we prove here is \cref{lem:basicnum}, a slight variant of the above lemma, which we restate below for convenience.

\basicnum*


\begin{proof}
By the assumption on $\fn$, there exists some $\eps > 0$ such that
\begin{equation}
\point - 2\coef \point^{1+\rexp}
	\leq \fn(\point)
	\leq \point - \half[\coef] \point^{1+\rexp}
	\quad
	\text{for all $\point \in [0,\eps]$}.
\end{equation}
Note first that, if $\init[\seq] \leq \eps$, \cref{lem:Polyak} readily implies that $\curr[\seq]$ converges to $0$ and that $\curr[\seq] \leq \eps$ for all\;$\run$.
Moreover, if $\eps$ is small enough so that $1 -2\coef \eps^\rexp > 0$ and $\init[\seq]$ is positive, this implies that all $\curr[\seq]$, for $\run = \running$, are positive.
In particular, 
we consider the sequence $\curr[\seq]^{-\rexp}$, $\run = \running$, for which we\;get
\begin{align}
\next[\seq]^{-\rexp} - \curr[\seq]^{-\rexp}
	&= \bracks[\big]{\curr[\seq] - \coef \curr[\seq]^{1+\rexp} + o\parens[\big]{\curr[\seq]^{1+\rexp}}}^{-\rexp}
		- \curr[\seq]^{-\rexp}
	\notag\\
	&= \curr[\seq]^{-\rexp}(1 - \coef \curr[\seq]^{\rexp} + o\left(\curr[\seq]^{\rexp}\right))^{-\rexp}
		- \curr[\seq]^{-\rexp}
	= \rexp \coef + o\left(1\right)\,.
\end{align}
Hence, $\curr[\seq]^{-\rexp} \sim \rexp \coef \run$ which gives the result.
\end{proof}

\subsection{Properties of Bregman divergences and the induced prox-mappings}
We recall some basic properties of the Bregman divergence and the induced prox-mapping.
Variants of these properties are fairly well known in the literature, so we omit their proofs and we refer the interested reader to \cite{BecTeb03,JNT11,MZ19,MLZF+19,DMSV23,MHC24} and references therein for a more detailed discussion.
In particular, \cref{lem:mirror,lem:threepoint,lem:proxlip} correspond to Lemmas B.1, B.2 and B.4(a) of \cite{MLZF+19}, respectively.

\begin{restatable}{lemma}{mirror}
\label{lem:mirror}
Let $\hreg$ be a Bregman regularizer on $\points$ and let $\subsel\hreg$ be a continuous selection of $\subd\hreg$.
Then, for all $\point\in\proxdom$, $\new\in\points$ and $\dvec\in\dpoints$, we have:
\begin{subequations}
\begin{flalign}
\quad
	a)\;\;
	&\subd\hreg(\point)
	= \subsel\hreg(\point)
		+ \pcone(\point)
	\\
\quad
	b)\;\;
	&\new
	= \proxof{\point}{\dvec}
	\iff
\subsel\hreg(\point) + \dvec
	\in \subd\hreg(\new)
	\iff
\subsel\hreg(\new) - \subsel\hreg(\point)
	\in \dvec - \pcone(\new)
\end{flalign}
\end{subequations}
where $\pcone(\point) = \setdef{\dvec\in\dpoints}{\braket{\dvec}{\pointalt - \point} \leq 0 \; \text{for all $\pointalt\in\points$}}$ denotes the polar cone to $\points$ at $\point$.
\end{restatable}

\begin{restatable}[$3$-point identity]{lemma}{threepoint}
\label{lem:threepoint}
For all $\base\in\points$ and all $\point,\new\in\proxdom$, we have:
\begin{equation}
\label{eq:threepoint}
\breg(\base,\new)
	= \breg(\base,\point)
		+ \breg(\point,\new)
		+ \braket{\subsel\hreg(\new) - \subsel\hreg(\point)}{\point - \base}
\end{equation}
\end{restatable}

\begin{restatable}[Non-expansiveness]{lemma}{proxlip}
\label{lem:proxlip}
For all $\point\in\proxdom$ and all $\dvec,\new[\dvec]\in\dpoints$ we have:
\begin{equation}
	\norm{\proxof{\point}{\new[\dvec]} - \proxof{\point}{\dvec}}
	\leq \dnorm{\new[\dvec] - \dvec}
\end{equation}
\end{restatable}

The next two results that we provide consider the evolution of the Bregman divergence before and after a prox step (or two);
they are both adapted from \cite[Proposition B.3]{MLZF+19}, with the added proviso that $\hreg$ is assumed to be $1$-strongly convex on $\nhdalt$ (as per \cref{rem:Bregman}).

\begin{lemma}
\label{lem:onestep-app}
Let $\new = \proxof{\point}{\dvec}$ for $\point\in\proxdom$, $\dvec\in\dpoints$ such that $\new$ is still in $\nhdalt$.
Then, for all $\base\in\points$, 
$\dbase\in\pcone(\base)$, we have:\!\!
\begin{subequations}
\begin{align}
\breg(\base,\new)
	&\leq \breg(\base,\point)
		+ \braket{\dvec - \dbase}{\new - \base}
		- \breg(\new,\point)
		\\
	&\leq \breg(\base,\point)
		+ \braket{\dvec - \dbase}{\point - \base}
		+ \tfrac{1}{2} \dnorm{\dvec - \dbase}^{2}\,.
\end{align}
\end{subequations}
\end{lemma}

\begin{proof}
Our proof follows \cite[Proposition B.3]{MLZF+19}, but with a slight modification to account for the extra term 
with $\dbase\in\pcone(\base)$.
The first step is to invoke the three-point identity \eqref{eq:threepoint} to write
\begin{equation}
\breg(\base,\point)
	= \breg(\base, \new)
		+ \breg(\new, \point)
		+ \braket{\subsel\hreg(\point) - \subsel\hreg(\new)}{\new - \base}.
\end{equation}
Then, after rearranging to isolate $\breg(\base,\new)$, we get
\begin{align}
\breg(\base,\new)
	&= \breg(\base, \point)
		- \breg(\new, \point)
		- \braket{\subsel\hreg(\point) - \subsel\hreg(\new)}{\new - \base}
	\notag\\
	&\leq \breg(\base,\point)
		- \breg(\new, \point)
		+ \braket{\dvec}{\new - \base}
\end{align}
where the inequality in the last line follows from \cref{lem:mirror}.
Hence, given that $\braket{\dbase}{\new - \base} \leq 0$ by the fact that $\dbase\in\pcone(\base)$, we readily obtain
\begin{equation}
\label{eq:onestep-proof}
\breg(\base, \new)
	\leq  \breg(\base, \point)
		- \breg(\new, \point)
		+ \braket{\dvec - \dbase}{\new - \base}\,.
\end{equation}

For the second inequality of the lemma, note that 
\begin{align}
\braket{\dvec - \dbase}{\new - \base}
	&= \braket{\dvec - \dbase}{\point - \base}
		+ \braket{\dvec - \dbase}{\new - \point}
	\notag\\
	&\leq \braket{\dvec - \dbase}{\point - \base}
		+ \tfrac{1}{2} \dnorm{\dvec - \dbase}^{2}
		+ \tfrac{1}{2} \norm{\new - \point}^{2}
	\notag\\
	&\leq \braket{\dvec - \dbase}{\point - \base}
		+ \tfrac{1}{2} \dnorm{\dvec - \dbase}^{2}
		+ \breg(\new,\point)
\end{align}
where the penultimate inequality follows directly from Young's inequality and the last one from \eqref{eq:Breg-lower} and the fact that $\point$ and $\new$ both lie in $\nhdalt$.
Our assertion is then obtained by combining this last bound with \eqref{eq:onestep-proof}.
\end{proof}

\begin{lemma}
\label{lem:twostep-app}
Let $\new_{i} = \proxof{\point}{\dvec_{i}}$ for some $\point\in \nhdalt \cap \proxdom$ and $\dvec_{i}\in\dpoints$ such that $\new_i \in \nhdalt$, $i=1,2$.
Then, for all $\base\in\points$ and all $\dbase\in\pcone(\base)$, we have:
\begin{equation}
\breg(\base,\new_{2})
	\leq \breg(\base,\point)
		+ \braket{\dvec_{2} - \dbase}{\new_{1} - \base}
		+ \tfrac{1}{2} \dnorm{\dvec_{2} - \dvec_{1} - \dbase}^{2}
		- \tfrac{1}{2} \norm{\new_{1} - \point}^{2}.
\end{equation}
\end{lemma}

\begin{proof}
Our proof follows \citep[Proposotion B.4]{MLZF+19}, again with a slight modification to account for the extra terms with 
$\dbase\in\pcone(\base)$.
Specifically, applying \cref{lem:onestep} with $\new_{2} = \mprox_\point(\dpoint_{2})$ and $\dbase \in \pcone(\base)$ gives
\begin{align}
\label{eq:twostep-proof}
\breg(\base,\new_{2})
	&\leq \breg(\base,\point)
		+ \braket{\dpoint_{2} - \dbase}{\new_{2} - \base}
		- \breg(\new_{2}, \point)
	\notag\\
	&\leq \breg(\base, \point)
		+ \braket{\dpoint_{2} - \dbase}{\new_{1} - \base}
		+ \braket{\dpoint_{2} - \dbase}{\new_{2} - \new_{1}}
		- \breg(\new_{2},\point)
\end{align}
To lower bound $\breg(\new_{2}, \point)$, we use again \cref{lem:onestep} with $\base \gets \new_{2}$ and $\new_{1} = \mprox_\point(\dpoint_{1})$.
This readily gives
\begin{equation}
\breg(\new_{2},\new_{1})
	\leq \breg(\new_{2},\point)
		+ \braket{\dpoint_{1}}{\new_{1} - \new_{2}}
		- \breg(\new_{1},\point)
\end{equation}
and hence,
after rearranging the above to isolate $\breg(\new_{2},\point)$ and substituting the resulting bound in \eqref{eq:twostep-proof}, we get
\begin{equation}
\breg(\base,\new_{2})
	\leq \breg(\base,\point)
		+ \braket{\dpoint_{2} - \dbase}{\new_{1} - \base}
		+ \braket{\dpoint_{2} - \dpoint_{1} - \dbase}{\new_{2} - \new_{1}} 
		- \breg(\new_{2},\new_{1})
		- \breg(\new_{1}, \point).
\end{equation}
Thus, by Young's inequality and the local strong convexity of $\hreg$, we finally obtain
\begin{align}
\breg(\base,\new_{2})
	&\leq \breg(\base, \point)
		+ \braket{\dpoint_{2} - \dbase}{\new_{1} - \base}
		+ \tfrac{1}{2}\dnorm{\dpoint_{2} - \dpoint_{1} - \dbase}^{2}
	\notag\\
	&\qquad
		+ \tfrac{1}{2}\norm{\new_{2} - \new_{1}}^{2}
		- \tfrac{1}{2}\norm{\new_{2}-\new_{1}}^{2}
		- \tfrac{1}{2}\norm{\new_{1}-\point}^{2}
	\notag\\
	&\leq \breg(\base,\point)
		+ \braket{\dpoint_{2} - \dbase}{\new_{1} - \base}
			+ \tfrac{1}{2}\dnorm{\dpoint_{2} - \dpoint_{1} - \dbase}^{2}
			- \tfrac{1}{2}\norm{\new_{1}-\point}^{2}
\end{align}
and our proof is complete.
\end{proof}

\subsection{Legendre exponent for interior points}

We now proceed to provide a more formal footint to our discussion in \Cref{sec:Legendre} regarding the fact that $\legof{\base}=0$ whenever $\base$ is an interior point.
The formal statement is as follows.

\begin{lemma}
\label{lem:Leg-proxdom}
Suppose that $\subsel\hreg$ is locally Lipschitz continuous.
Then $\legof{\base} = 0$ for all $\base \in \proxdom$;
in particular, $\legof{\base} = 0$ whenever $\base\in\relint\points$.
\end{lemma}

\begin{proof}
Fix some $\base \in \proxdom$ and suppose that $\subsel\hreg$ is locally Lipschitz continuous.
Then there exists a neighborhood $\legnhd$ of $\base$ in $\points$ and some $\legconst > 0$ such that
\begin{equation}
\label{eq:hLip}
\dnorm{\nabla \hreg(\base) - \nabla \hreg(\point)}
	\leq \legconst \norm{\base - \point}
	\quad
	\text{for all $\point \in \legnhd \cap \proxdom$}.
\end{equation}
Now, since $\nabla \hreg(\base) \in \subd \hreg(\base)$,
we also have
\begin{align}
\breg(\base,\point)
	= \hreg(\base) - \hreg(\point) - \braket{\nabla \hreg(\point)}{\base - \point}
	&\leq \braket{\nabla \hreg(\base) - \nabla \hreg(\point)}{\base - \point}
	\notag\\
	&\leq \dnorm{\nabla \hreg(\base) - \nabla \hreg(\point)} \norm{\base - \point}
	\leq \legconst \norm{\base - \point}^{2}
\end{align}
for all $\point\in\legnhd \cap \proxdom$.
This shows that \eqref{eq:Breg-local} holds with $\legexp=0$, \ie $\legof{\base}=0$.
\end{proof}

\subsection{A separation result}

We now proceed to prove \cref{lem:separation}, which we restate below for convenience:

\separation*


\begin{proof}
Our claim is trivial if $\coords = \actcoords$, so we will focus exclusively on the case $\coords \subsetneq \actcoords$.
The stated constant $\polycst = \polycst(\mat, \cvec, \sol)$ will then be obtained as the maximum of $1$ and the constants we obtain for each possible $\coords \subsetneq \actcoords$.
 
The proof consists in discussing whether there exists 
 $(\coef_\coord)_{\coord \in \actcoords\setminus\coords} \in (\R_{+})^{\actcoords\setminus \coords}$ \emph{not all zero} and $(\coefalt_\coord)_{\coord \in \coords} \in \R^{\coords}$ such that the inclusion
\begin{align}
\points
	\subseteq \left\{\point \in \R^\nCoords : \sum_{\coord \in \actcoords \setminus\coords} \coef_\coord \point_\coord =\sum_{\coord \in \coords}\coefalt_\coord \point_\coord \right\} \label{eq:inclusion}
\end{align}
holds.
\Cref*{itm:inactive} considers the case when such coefficients exist, while \cref*{itm:active} considers when this is not possible.

\para{\cref*{itm:inactive}}
Assume that there exists $(\coef_\coord)_{\coord \in \actcoords\setminus\coords} \in (\R_{+})^{\actcoords\setminus \coords}$ \emph{not all zero} and $(\coefalt_\coord)_{\coord \in \coords} \in \R^{\coords}$ such that \eqref{eq:inclusion} holds.
In this case, $\coords$ must be non-empty since otherwise $\points$ would be reduced to $\{0\}$ (see the first inclusion), violating the definition \eqref{eq:polyhedron} of $\points$.
In addition, there is some $\coord \in \actcoords \setminus \coords$ such that $\coef_\coord > 0$ and thus we have
\begin{align}
 \forall \point \in \points,\,
\point_\coord \leq \frac{\max\left(|\coef_\coordalt|: \coordalt \in \coords\right)}{\coef_\coord}\max(\point_\coordalt : \coordalt \in \coords)\,
\end{align}
which corresponds to the first case of the lemma.

\para{\cref*{itm:active}}
For all $(\coef_\coord)_{\coord \in \actcoords\setminus\coords} \in (\R_{+})^{\actcoords\setminus \coords}$ \emph{not all zero} and $(\coefalt_\coord)_{\coord \in \coords} \in \R^{\coords}$, \eqref{eq:inclusion} does not hold.
To interpret this situation, we use the fact that $\points$ is of the general polyhedral form \eqref{eq:polyhedron} so $\aff \points = \sol + \ker \mat$  and $\sol$ always satisfies $\sum_{\coord \in \actcoords \setminus\coords} \coef_\coord \sol_\coord =\sum_{\coord \in \coords}\coefalt_\coord \sol_\coord = 0$ so that
\begin{align}
\eqref{eq:inclusion}
	&\iff \aff \points \subset\!\left\{\point \in \R^\nCoords\!:\!\!\sum_{\coord \in \actcoords \setminus\coords} \coef_\coord \point_\coord =\sum_{\coord \in \coords}\coefalt_\coord \point_\coord \right\}
	\notag\\
	&\iff \ker \mat \subset\!\left\{\point \in \R^\nCoords\!:\!\!\sum_{\coord \in \actcoords \setminus\coords} \coef_\coord \point_\coord =\sum_{\coord \in \coords}\coefalt_\coord \point_\coord \right\}
	\notag\\
	&\iff \sum_{\coord \in \actcoords \setminus \coords} \coef_\coord \bvec_\coord - \sum_{\coord \in \coords} \coefalt_\coord \bvec_\coord \in \row(\mat).
\end{align}
Therefore, the fact that  \eqref{eq:inclusion} does not hold for all $(\coef_\coord)_{\coord \in \actcoords\setminus\coords} \in (\R_{+})^{\actcoords\setminus \coords}$ \emph{not all zero} and $(\coefalt_\coord)_{\coord \in \coords} \in \R^{\coords}$ means that,
the system,
\begin{align}
\sum_{\coord \in \actcoords \setminus \coords} \coef_\coord \bvec_\coord = \sum_{\coord \in \coords} \coefalt_\coord \bvec_\coord + \mat^{\top} \pvecalt\,,
\end{align}
with variables
$
(\coef_\coord)_{\coord \in \actcoords \setminus \coords} \in (\R_{+})^{\actcoords \setminus \coords} \text{not all zero, } (\coefalt_\coord)_{\coord \in\coords} \in \R^{\coords}, \pvecalt \in \R^{\nConstr},
$ has no solution.
Hence, by Motzkin's theorem on the alternative (see \eg \cite[\S1.4.2]{BN01})\footnote{With the notations of \cite[\S1.4.2]{BN01}, the lines of the matrix $S$ are made of the $\bvec_\coord$ for $\coord \in \actcoords \setminus \coords$ and the lines of the matrix $N$ are the $\bvec_\coord$ for $\coord \in \coords$, $-\bvec_\coord$ for $\coord \in \coords$, the lines of $\mat$ and their opposite.}, this means that the system
\begin{equation}
\begin{cases}
	\pvec_{\coord} > 0
		&\quad
		\text{for $\coord \in \actcoords \setminus \coords$}
		\\
	\pvec_{\coord} = 0
		&\quad
		\text{for $\coord \in \actcoords$}
		\\
	\mat\pvec = 0
\end{cases}
\end{equation}
admits a solution $\pvec \in \R^{\nCoords}$.
Rescaling 
$\pvec$ and setting $\polycst$ to $\max(\norm{\pvec}, \norm{\pvec}_\infty)$ then gives the second case.
\end{proof}

\section{Omitted calculations}
\label{app:ex}
In this appendix, we provide some computational details that were left out of the main text to streamline our presentation.

\begin{example}[name=Hellinger distance,continues=ex:Hell]
We proceed to compute the Taylor expansion of $\fixmap$ near $\sol = -1$ for the shifted operator $\vecfield(\point) = \point+1$.
Indeed, in this case, the fixed point operator $\fixmap$ is given by
\begin{align}
\label{eq:app-fixmap}
\fixmap(\point)
	= \proxof{\point}{-\step \vecfield(\point)}
	&= \proxof{\point}{-\step (\point+1)}
	\notag\\
	&= \frac{\point - \step(\point+1)\sqrt{1-\point^{2}}}{\sqrt{1-\point^{2} + (\point -\step(\point+1)\sqrt{1-\point^{2}})^{2}}}
	\notag\\
	&= \frac{\fixmapalt(\point)}{\sqrt{1-\point^{2} +\fixmapalt(\point)^{2}}},
\end{align}
with $\fixmapalt(\point) = \point - \step(\point+1)\sqrt{1-\point^{2}}$.
Now, the behavior of $\fixmapalt$ near $\sol = -1$ can be approximated as
\begin{align}
\label{eq:app-fixmapalt}
\fixmapalt(\point)
	&=\point - \step(\point+1)^{3/2}(1-\point)^{1/2}
	\notag\\
	&= -1 + (\point+1) - \step(\point+1)^{3/2}(2 - (\point+1))^{1/2}
	\notag\\
	&= -1 + (\point+1) - \sqrt 2 \step(\point+1)^{3/2}\left(1 - \tfrac{1}{4}\parens*{\point+1} + o(\point+1)\right)
	\notag\\
	&= -1 + (\point+1) - \sqrt 2 \step(\point+1)^{3/2} + \tfrac{\sqrt 2 \step}{4}(\point+1)^{5/2} + o \parens*{(\point+1)^{5/2}}\,.
\end{align}
Another Taylor expansion then yields
\begin{align}
\fixmapalt(\point)^{2}
	&= \parens*{1 - (\point+1) + \sqrt 2 \step(\point+1)^{3/2} - \tfrac{\sqrt 2 \step}{4}(\point+1)^{5/2} + o \parens*{(\point+1)^{5/2}}}^{2}
	\notag\\
	&= 1 - 2(\point+1) + 2\sqrt 2 \step(\point+1)^{3/2} + (\point+1)^{2} - \tfrac{\sqrt 2 \step}{2}(\point+1)^{5/2} + o \parens*{(\point+1)^{5/2}}
\end{align}
so the denominator of \cref{eq:app-fixmap} becomes
\begin{align}
\sqrt{1-\point^{2} +\fixmapalt(\point)^{2}}
	&= \parens*{(\point+1)(2 - (\point+1) + \fixmapalt(\point)^{2}}^{2}
	\notag\\
	&= \parens*{1 + 2\sqrt 2 \step(\point+1)^{3/2} - \tfrac{\sqrt 2 \step}{2}(\point+1)^{5/2} + o \parens*{(\point+1)^{5/2}}}^{2}
	\notag\\
	&= 1 + \sqrt 2 \step(\point+1)^{3/2} - \tfrac{\sqrt 2 \step}{4}(\point+1)^{5/2} + o \parens*{(\point+1)^{5/2}}\,.
\end{align}
Thus, plugging this expansion and \cref{eq:app-fixmapalt} into \cref{eq:app-fixmap} gives
\begin{align}
\fixmap(\point)
	&= \frac
		{-1 + (\point+1) - \sqrt 2 \step(\point+1)^{3/2} + \tfrac{\sqrt 2 \step}{4}(\point+1)^{5/2} + o \parens*{(\point+1)^{5/2}}}
	{1 + \sqrt 2 \step(\point+1)^{3/2} - \tfrac{\sqrt 2 \step}{4}(\point+1)^{5/2} + o \parens*{(\point+1)^{5/2}}}
	\notag\\
	&= \parens*{
 -1 + (\point+1) - \sqrt 2 \step(\point+1)^{3/2} + \tfrac{\sqrt 2 \step}{4}(\point+1)^{5/2} + o \parens*{(\point+1)^{5/2}}
 }
	\notag\\
	&\qquad\times \parens*{
1 - \sqrt 2 \step(\point+1)^{3/2} + \frac{\sqrt 2 \step}{4}(\point+1)^{5/2} + o \parens*{(\point+1)^{5/2}}
 }
	\notag\\
&=
		-1
		+ (\point+1)
		- 2\sqrt{2}\step (\point+1)^{5/2} + o\parens*{(\point+1)^{5/2}}\,,
\end{align}
which gives our assertion when $\sol = -1$.
\hfill
\endenv
\end{example}

\begin{example}[name=Three-dimensional simplex,continues=ex:simplex-2d]
We conclude our treatment of the simplex by showing that $\state_{2,\run} \sim \state_{2,\run} / \state_{3,\run} = \Omega(1/\run)$ if $\slack_{2} = 0$ but $\slack_{1} > 0$.
To begin with, we have $\vecfield_{2}(\curr) = \state_{2,\run} = o(1)$ so, arguing as in the first part of the example, we readily get
\begin{equation}
\frac{\state_{1,\run+1}}{\state_{2,\run+1}}
	= \frac{\state_{1,\run}}{\state_{2,\run}}
		\exp(-\step\slack_{1} + o(1)),
\end{equation}
so $\state_{1,\run} / \state_{2,\run}$ converges to $0$ at a geometric rate.
Accordingly, the quantity 
$\state_{3,\run} / \state_{2,\run}$ 
is bounded as
\begin{align}
\frac{\state_{3,\run+1}}{\state_{2,\run+1}}
	&= \frac{\state_{3,\run}}{\state_{2,\run}}
		\exp \parens*{\step \vecfield_{2}(\curr) - \step \vecfield_{3}(\curr)}
	= \frac{\state_{3,\run}}{\state_{2,\run}}
		\exp \parens*{\step \state_{2,\run} - \step (\state_{3,\run} - 1)}
	\notag\\
    &= \frac{\state_{3,\run}}{\state_{2,\run}}
		\exp \parens*{2\step \state_{2,\run} + \step \state_{1,\run}}
	\leq
		\frac{\state_{3,\run}}{\state_{2,\run}}
		\exp \parens*{2\step \frac{\state_{2,\run}}{\state_{3,\run}} + \step \state_{1,\run}}
\end{align}
Now, since both $\frac{\state_{2,\run}}{\state_{3,\run}}$ and $\state_{1,\run}$ go to zero, 
\begin{align}
\frac{\state_{3,\run+1}}{\state_{2,\run+1}}
	&\leq \frac{\state_{3,\run}}{\state_{2,\run}}
		\parens*{1
			+ 2\step \frac{\state_{3,\run}}{\state_{2,\run}}
			+ \step \state_{1,\run} 
			+ o\parens*{2\frac{\state_{3,\run}}{\state_{2,\run}}
			+ \state_{1,\run}}}
	\notag\\
&=
		\frac{\state_{3,\run}}{\state_{2,\run}}
			+2\step
			+o \parens*{1}\,.
\end{align}
since $\state_{1,\run}/\state_{2,\run}$ vanishes as $\run\to\infty$.
Hence, after telescoping, we conclude that $\frac{\state_{3,\run}}{\state_{2,\run}}
	\leq 2 \step \run + o(\run),$ which in turn shows that $\state_{2,\run} \sim \state_{2,\run} / \state_{3,\run} = \Omega(1/\run)$, as claimed.
\hfill
\endenv
\end{example}

\normalsize
\bibliographystyle{siamplain}
\bibliography{bibtex/IEEEabrv,bibtex/Bibliography-PM}

\end{document}